\documentclass[11pt]{amsart}

\usepackage{lineno,hyperref}
\usepackage[utf8]{inputenc}

\usepackage{amssymb}
\usepackage{amsmath}
\usepackage{mathrsfs}

\usepackage{comment}

\newtheorem{theorem}{Theorem}[section]
\newtheorem{lemma}[theorem]{Lemma}
\newtheorem{proposition}[theorem]{Proposition}
\newtheorem{corollary}[theorem]{Corollary}
\newtheorem{definition}[theorem]{Definition}
\newtheorem{example}[theorem]{Example}

\newtheorem{remark}[theorem]{Remark}

\newcommand{\N}{\mathbb{N}}
\newcommand{\Z}{\mathbb{Z}}
\newcommand{\Q}{\mathbb{Q}}
\newcommand{\R}{\mathbb{R}}

\newcommand{\ns}[1]{\,\!^\ast#1} %comandi sinistri
\newcommand{\sh}[1]{\,\!^\circ#1} %comandi sinistri
\renewcommand{\ni}[2]{[#1,#2]_{\Lambda}}

\newcommand{\leb}{\mu_L}
\newcommand{\B}{\test'({\dom})}

\newcommand{\F}{F_{\Lambda}}
\newcommand{\weakly}{\rightharpoonup}

\newcommand{\hR}{\ns{\R}} %comandi sinistri
\newcommand{\fin}{\hR_{fin}}
\renewcommand{\Lambda}{\mathbb{X}}

\newcommand{\D}{\mathbb{D}}

\newcommand{\test}{\mathscr{D}_\Lambda}
\newcommand{\tests}{\mathscr{D}}
\newcommand{\supp}{\mathrm{supp\,}}
\renewcommand{\L}{{L_{\Lambda}}}

\newcommand{\J}{J_{\Lambda}}
\newcommand{\f}{T_{\Lambda}}
\newcommand{\lp}[1]{\widehat{#1}}
\newcommand{\lpi}{P}

\newcommand{\grid}[1]{\mathbb{G}({#1})}

\renewcommand{\sim}{\approx}

\newcommand{\dom}{\Omega_{\Lambda}}

\newcommand{\norm}[1]{\Vert#1\Vert}
\newcommand{\bcf}{C^0_b}
\newcommand{\rad}{\mathbb{M}}
\newcommand{\prob}{\rad^{\mathbb{P}}}

\newcommand{\stdiv}{\mathrm{div}}
\renewcommand{\div}{\stdiv_{\Lambda}}

\newcommand{\grad}{\nabla_\Lambda}
\newcommand{\lap}{\Delta_\Lambda}

\newcommand{\ldual}{\langle}
\newcommand{\rdual}{\rangle_{\tests(\Omega)}}

\renewcommand{\Lambda}{\mathbb{X}}
\renewcommand{\F}{F_{\Lambda}}

\renewcommand{\B}{\test'({\dom})}
\newcommand{\M}{L^{\dagger}}

\begin{document}
	
	%\begin{frontmatter}
		
		\title[Grid functions, distributions, and PDEs]{Grid functions of nonstandard analysis in the theory of distributions and in partial differential equations}
		
		\author{Emanuele Bottazzi}
		\address{University of Trento, Italy}
		%\ead[url]{https://sites.google.com/view/emanuele-bottazzi/home}
		
		
		%\address[mymainaddress]{University of Trento, Italy}
		
		\maketitle
		
		\begin{abstract}
			We introduce the space of grid functions, a space of generalized functions of nonstandard analysis that provides a coherent generalization both of the space of distributions and of the space of Young measures. We will show that in the space of grid functions it is possible to formulate problems from many areas of functional analysis in a way that coherently generalizes the standard approaches. As an example, we discuss some applications of grid functions to the calculus of variations and to the nonlinear theory of distributions. Applications to nonlinear partial differential equations will be discussed in a subsequent paper. 
		\end{abstract}
		
		%\begin{keyword}
			%generalized functions\sep Young measures\sep nonstandard analysis
			%\MSC[2010] 46F30\sep \sep 26E35 \sep 03H05
		%\end{keyword}
		
	%\end{frontmatter}
\tableofcontents

\section{Introduction}

The theory of distributions, pioneered by Dirac in \cite{dirac} and developed in the first half of the XX Century, has become one of the fundamental tools of functional analysis.
In particular, the possibility to define the weak derivative of a non-differentiable function has allowed the formulation and the study of a wide variety of nonsmooth phenomena by the theory of partial differential equations.
However, the lack of a nonlinear theory of distributions is a limiting factor both for the applications and for the theoretical study of nonlinear PDEs.
On the one hand, in the description of some physical phenomena such as shock waves and relativistic fields, it arises the need to have some mathematical objects which cannot be formalized in the sense of distributions (we refer to \cite{colombeau advances} for some examples).
On the other hand, the absence of a nonlinear theory of distributions poses some limitations in the study of nonlinear partial differential equations: while some nonlinear problems can be solved by studying the limit of suitable regularized problems, other problems do not allow for solutions in the sense of distributions (see for instance the discussion in \cite{evans nonlinear}).

In 1954, L.\ Schwartz proved that the absence of a nonlinear theory for distributions is intrinsic: more formally, the main theorem of \cite{schwartz} entails that there is no differential algebra $(A, +, \otimes, D)$ in which the real distributions $\mathcal{D}'$ can be embedded and the following conditions are satisfied:
\begin{enumerate}
	\item $\otimes$ extends the product over $C^0$ functions;
	\item $D$ extends the distributional derivative $\partial$;
	\item the product rule holds: $D(u\otimes v) = (Du)\otimes v + u\otimes(Dv)$.
\end{enumerate}

Despite this negative result, there have been many attempts at defining some notions of product between distributions (see for instance \cite{neutrix, survey}).
Following this line of research, Colombeau in 1983 proposed an organic approach to a theory of generalized functions \cite{colombeau 1983}: Colombeau's idea is to embed the distributions in a differential algebra with a good nonlinear theory, but at the cost of sacrificing the coherence between the product of the differential algebra with the product over $C^0$ functions.
This approach has been met with interest and has proved to be a prolific field of research.
For a survey of the approach by Colombeau and for recent advances, we refer to \cite{colombeau advances}.

Research about generalized functions beyond distributions is also being carried out within the setting of nonstandard analysis.
Possibly the earliest result in this sense is the proof by Robinson that the distributions can be represented by smooth functions of nonstandard analysis and by polynomials of a hyperfinite degree \cite{nsa robinson}.
Distributions have also been represented by functions defined on hyperfinite domains, for instance by Kinoshita in \cite{moto} and, with a different approach, by Sousa Pinto and Hoskins in \cite{hyperfinite pinto}.
Another nonstandard approach to the theory of generalized functions has been proposed by Oberbuggenberg and Todorov in \cite{oberbuggenberg} and further studied by Todorov et al.\ \cite{todorov2, todorov3}.
In this approach, the distributions are embedded in an algebra of asymptotic functions defined over a Robinson field of asymptotic numbers.
Moreover, this algebra of asymptotic functions can be seen as a generalized Colombeau algebra where the set of scalars is an algebraically closed field rather than a ring with zero divisors.
In this setting, it is possible to study generalized solutions to differential equations, and in particular to those with nonsmooth coefficient and distributional initial data \cite{nonsmo2, nonsmo1}.

Another theory of generalized functions oriented towards the applications in the field of partial differential equations and of the calculus of variations has been developed by Benci and Luperi Baglini.
In \cite{ultrafunctions1} and subsequent papers \cite{benci, ultramodel, ultraschwartz, ultraapps}, the authors developed a theory of ultrafunctions, i.e.\ nonstandard vector spaces of a hyperfinite dimension that extend the space of distributions.
In particular, the space of distributions can be embedded in an algebra of ultrafunctions $V$ such that the following inclusions hold: $\tests'(\R) \subset V \subset \ns{C^1(\R)}$ \cite{ultraschwartz}.
%By taking into account the inclusion $\ns{C^1(\R) \subset \tests'(\R)$, 
This can be seen as a variation on a result by Robinson and Bernstein, that in \cite{invariant} showed that any Hilbert space $H$ can be embedded in a hyperfinite dimensional subspace of $\ns{H}$.
In the setting of ultrafunctions, some partial differential equations can be formulated coherently by a Galerkin approximation, while the problem of finding the minimum of a functional can be turned to a minimization problem over a formally finite vector space.
For a discussion of the applications of ultrafunctions to functional analysis, we refer to \cite{ultrafunctions1, ultramodel, ultraapps}.

The idea of studying the solutions to a partial differential equation via a hyperfinite %system of ordinary differential equations in the setting of nonstandard analysis
Galerkin approximation is not new.
For instance, Capi\'{n}sky and Cutland in \cite{capicutland statistic} studied statistical solutions to parabolic differential equations by discretizing the equation in space by a Galerkin approximation in an hyperfinite dimension.
The nonstandard model becomes then a hyperfinite system of ODEs that, by transfer, has a unique nonstandard solution.
From this solution, the authors showed that it is possible to define a standard weak %and statistical
solution to the original problem.
In the subsequent \cite{capicutland1}, the authors proved the existence of weak and statistical solutions to the Navier-Stokes equations in 3-dimensions by modelling the equations with a similar hyperfinite Galerkin discretization in space.
This approach has spanned a whole line of research on the Navier-Stokes equations, concerning both the proof of the existence of solutions (see for instance \cite{capicutland noise, cutland n+1}) and the definition and the existence of attractors (see for instance \cite{capicutland attractors, cutland attractors}).
One of the advantages of this approach is that, by a hyperfinite discretization in space, the nonstandard models have a unique global solution, even when the original problem does not.
For a discussion of the relation between the uniqueness of the solutions of the nonstandard formulation and the non-uniqueness of the weak solutions of the original problem in the case of the Navier-Stokes equations, we refer to \cite{capicutland1}.

In the theoretical study of nonlinear partial differential equations, sometimes problems do not allow even for a weak solution.
However, the development of the notion of Young measures, originally introduced by L.\ C.\ Young in the field of optimal control
in \cite{young1}, %and their subsequent development allowed
has allowed for a synthetic characterization of the behaviour of the weak-$\star$ limit of the composition between a nonlinear continuous function and a uniformly bounded sequence in $L^\infty$.
By enlarging the class of admissible solutions to include Young measures, one can define generalized solutions for some class of nonlinear problems as the weak-$\star$ limit of the solutions to a sequence of regularized problems \cite{demoulini, evans nonlinear, matete, matete2, plotnikov, slemrod, smarrazzo}.
A similar approach can be carried out in the field of optimal controls, where generalized controls in the sense of Young measures can be defined as the measure-valued limit points of a minimizing sequence of controls.
For an in-depth discussion of the role of Young measures as generalized solution to PDEs and as generalized controls, we refer to \cite{balder, evans nonlinear, sychev, webbym}.
%See the introduction and the references in the lecture notes by Balder \cite{balder}.

In \cite{cutland controls3, cutland controls2, cutland controls}, Cutland showed that Young measures can be interpreted also as the standard part of internal controls of nonstandard analysis.
%In particular, the author studied problems from the calculus of variations by several nonstandard formulations, showing that nonstandard controls have a standard part that is a generalized control in the sense of Young measures .
The possibility to obtain a Young measure from a nonstandard control allows to study generalized solutions to nonlinear variational problems by means of nonstandard techniques: such an approach has been carried out for instance by Cutland in the aforementioned papers,  and by Tuckey in \cite{tuckey}.
For a discussion of this field of research, we refer to \cite{neves}.

\subsection*{Structure of the paper}

In this paper, we will discuss another theory of generalized functions of nonstandard analysis, hereafter called grid functions (see Definition \ref{def grid functions}), that provide a coherent generalization both of the space of distributions and of a space of parametrized measures that extends the space of Young measures.
In Section \ref{prelim}, we will define the space of grid functions, and recall some well-established nonstandard results that will be used throughout the paper.
In particular, we will formulate in the setting of grid functions some known results regarding the relations between the hyperfinite sum and the Riemann integral, and the finite difference operators of an infinitesimal step and the derivative of a $C^1$ function.

In Section \ref{distri}, we will study the relations between the grid functions and the distributions, with the aim of proving that every distribution can be obtained from a suitable grid function.
In order to reach this result, we will introduce an algebra of nonstandard test functions that can be seen as the grid function counterpart to the space $\tests(\Omega)$ of smooth functions with compact support over $\Omega \subseteq \R^k$.
By duality with respect to the algebra of test functions, we will define a module of grid distributions, and an an equivalence relation between grid functions (see Definition \ref{def equiv} and Definition \ref{def bounded grid functions}).
We will then prove that the set of equivalence classes of grid distributions with respect to this equivalence relation is a real vector space that is isomorphic to the space of distributions.
%The premises of this approach are similar to those by Kinoshita in \cite{moto}: in particular, our Theorem \ref{bello} can be obtained from Theorem 1 of \cite{moto} and from the properties of the equivalence relation defined at the beginning of Section \ref{distri}.
Afterwards, we will discuss how the finite difference operators generalize not only the usual derivative for $C^1$ functions, but also the distributional derivative.
%A similar result is also established by Kinoshita in \cite{moto} for a module of functions that is smaller than our module of bounded grid functions.
%On the other hand, this result provides the starting point for the work of Sousa Pinto and Hoskins on a hyperfinite representation of distributions in \cite{hyperfinite pinto}.
%They define a module of pre-distributions over $\R$ as the module of functions obtained by applying an arbitrary number of times the finite difference operator to some S-continuous function defined on a hyperfinite grid.
%Then, they define a module of global distributions over $\R$ by piecing together pre-distributions that agree on the intersection of their support.
%However, the authors %then study the properties of the discrete Fourier transform, and
%do not prove results comparable to our Theorem \ref{bello} or to Theorem 1 of \cite{moto}.

After having shown that the finite difference operator generalizes the distributional derivative, our study of the relations between grid functions and distributions concludes with a discussion of the Schwartz impossibility theorem.
In particular, %in Section \ref{section schwartz}
we will show that the space of distributions can be embedded in the space of grid functions in a way that
\begin{enumerate}
	\item the product over the grid functions generalizes the pointwise product between continuous functions;
	\item the finite difference is coherent with the distributional derivative modulo the equivalence relation induced by duality with test functions;
	\item a discrete chain rule for products holds.
\end{enumerate}
This theorem supports our claim that the space of grid functions provides a nontrivial generalization of the space of distributions.

In Section \ref{sez young}, we will embed the space of grid functions in the spaces $\ns{L^p}$ with $1 \leq p \leq \infty$, and we will study some properties of grid functions through this embedding.
Moreover, we will discuss a generalization of the embedding of $L^2(\Omega)$ in a hyperfinite subspace of $\ns{L^2}(\Omega)$ due to Robinson and Bernstein \cite{invariant}.
%This \textbf{embedding} is due to Robinson and Bernstein \cite{invariant}, and it is obtained by TECNICHE\ldots
This classic result will be generalized in two directions:
\begin{enumerate}
	\item for every $1 \leq p \leq \infty$, we will embed the spaces $L^p(\Omega)$ in the space of grid functions, which is a subspace of $\ns{L^p(\Omega)}$ of a hyperfinite dimension;
	\item the above embedding is actually an embedding of the bigger space $\tests'(\Omega)$ into a hyperfinite subspace of $\ns{L^p(\Omega)}$ for all $1 \leq p \leq \infty$.
\end{enumerate}
Moreover, this embedding is obtained with different techniques from the original result by Robinson and Bernstein.%, and in particular it does not make use of saturation.

In the second part of Section \ref{sez young}, we will establish a correspondence between grid functions and parametrized measures, in a way that is coherent with the isomorphism between equivalence classes of grid distributions and distributions discussed in Section \ref{distri}.
The results discussed in Section \ref{sez young} will be used in Section \ref{solutions}, where we will discuss the grid function formulation of partial differential equations, in Section \ref{selected applications}, where we will show selected applications of grid functions from different fields of functional analysis, and in the paper \cite{illposed}, where  we will study in detail a grid function formulation of a class of ill-posed partial differential equations with variable parabolicity direction.

In Section \ref{solutions}, we will discuss how to formulate partial differential equations in the space of grid functions in a way that coherently generalizes the standard notions of solutions.
In particular, stationary PDEs will be given a fully discrete formulation, while time-dependent PDEs will be given a continuous-in-time and discrete-in-space formulation, resulting in a hyperfinite system of ordinary differential equations, as in the nonstandard formulation of the Navier-Stokes equations by Capi\'{n}sky and Cutland.

%Since grid functions generalize both the distributions and the Young measures, they can be successfully applied to a variety of different problems: 
%In particular, at the end of this chapter we will discuss two classic problems in the nonlinear theory of distributions and in the calculus of variations, while in the next chapter we will discuss in detail a grid function formulation of a class of ill-posed partial differential equations with variable parabolicity direction.
In Section \ref{selected applications}, we will use the theory of grid functions developed so far to study two problems in the nonlinear theory of distributions and in the calculus of variations.
These problems are classically studied within different frameworks, but we will show that each of these problems can be formulated in the space of grid functions in a way that the nonstandard solutions generalize the respective standard solutions.

\section{Terminology and preliminary notions}\label{prelim}

In this section, we will now fix some notation and recall some results from nonstandard analysis that will be useful throughout the paper.

%The following definitions of standard objects are generalized as expected in the setting of nonstandard analysis.
If $A \subseteq \R^k$, then $\overline{A}$ is the closure of $A$ with respect to any norm in $\R^k$, $\partial A$ is the boundary of $A$, and $\chi_A$ is the characteristic function of $A$.
If $x \in \R$, then $\chi_x = \chi_{\{x\}}$.
If $f : A \rightarrow \R$%or if  $f : A \rightarrow \hR$
, $\supp f$ is the closure of the set $\{x \in A : f(x)\not=0\}$.
These definitions are generalized as expected also to nonstandard objects.

We consider the following norms over $\R^k$: if $x \in \R^k$ or $x \in \hR^k$, then $|x| = \sqrt{\sum_{i=1}^k x_i^2}$ is the euclidean norm, and $|x|_\infty = \max_{i = 1, \ldots, k} |x_i|$ is the maximum norm.

We will denote by $e_1, \ldots, e_k$ the canonical basis of $\R^k$. %, and we will denote by $\ns{\R}^{k}_+$ the non-negative elements of $\ns{\R^k}$, i.e.\ $\ns{\R^k_+} = \left\{ x \in \ns{\R}^{k} : x_i \geq 0 \mathrm{\ for \ } 1 \leq i \leq k \right\}$.
If $f : A \subseteq \R^m \rightarrow \R^k$, we will denote by $f_1, \ldots, f_k$ the hyperreal valued functions that satisfy the equality $f(x) = (f_1(x), \ldots, f_k(x))$ for all $x \in \hR$.

In the sequel, $\Omega \subseteq \R^k$ will be an open set.

If $f \in C^1(\Omega)$, we will denote the partial derivative of $f$ in the direction $e_i$ by $\frac{df}{dx_i}$ or $D_i f$.
If $\Omega \subseteq \R$, we will also write $f'$ for the derivative of $f$.
We adopt the multi-index notation for partial derivatives and, if $\alpha$ is a multi-index, we will denote by $D^\alpha f$ the function
%\begin{equation}\label{partialdev}
$$
D^\alpha f = \frac{\partial^{|\alpha|} f}{\partial x_1^{\alpha_1} \partial x_2^{\alpha_2}\ldots \partial x_k^{\alpha_k}}.
$$
%\end{equation}
If $\alpha = (\alpha_1, \ldots, \alpha_k)$ is a multi-index, then $\alpha-e_i = (\alpha_1, \ldots, \alpha_i -1, \ldots, \alpha_k)$.
If $f : [0,T]\times \Omega \rightarrow \R$, we will think of the first variable of $f$ as the time variable, denoted by $t$, and we will write $f_t$ for the derivative $\frac{\partial f}{\partial t}$.

We will often reference the following real vector spaces:
\begin{itemize}
	\item $\bcf(\R) = \{ f \in C^0(\R) : f \text{ is bounded and } \lim_{|x|\rightarrow \infty} f(x) = 0 \}$.
	\item $C^0_c(\Omega) = \{ f \in C^0_b(\Omega) : \supp f \subset \subset \Omega\}$.
	\item $\tests(\Omega) = \{ f \in C^\infty(\Omega) : \supp f \subset \subset \Omega\}$.
	\item A real distribution over $\Omega$ is an element of $\tests'(\Omega)$, i.e.\ a continuous linear functional $T : \tests(\Omega)\rightarrow\R$.
	If $T$ is a distribution and $\varphi$ is a test function, we will denote the action of $T$ over $\varphi$ by
	$ \ldual T, \varphi\rdual $.
	When $T$ can be identified with a $L^p$ function, we will sometimes write
	$
	\int_{\Omega} T\varphi dx
	$
	instead of $\ldual T, \varphi\rdual$.
	
	If $T \in \tests'(\R)$, we will denote the derivative of $T$ by $T'$ or $D T$.
	Recall that $T'$ is defined by the formula
	$
	\ldual DT, \varphi\rdual = - \ldual T, D\varphi\rdual.
	$
	If $T \in \tests'(\Omega)$ and $\alpha$ is a multi-index, the distribution $D^\alpha T$ is defined in a similar way:
	$
	\ldual D^\alpha T, \varphi\rdual = (-1)^{|\alpha|} \ldual T, D^\alpha \varphi\rdual.
	$
	\item In the sequel, measurable will mean measurable with respect to $\leb$, the Lebesgue measure over $\R^n$.
	Consider the equivalence relation given by equality almost everywhere: two measurable functions $f$ and $g$ are equivalent if $\leb(\{x\in\Omega : f(x) \not = g(x)\})=0$.
	We will not distinguish between the function $f$ and its equivalence class, and we will say that $f = g$ whenever the functions $f$ and $g$ are equal almost everywhere.
	
	For all $1 \leq p < \infty$, $L^p(\Omega)$ is the set of equivalence classes of measurable functions $f: \Omega \rightarrow \R$ that satisfy
	$$
	\int_{\Omega} |f|^p dx < \infty.
	$$
	If $f \in L^p(\Omega)$, the $L^p$ norm of $f$ is defined by $$\norm{f}_p^p = \int_{\Omega} |f|^p dx.$$
	
	$L^\infty(\Omega)$ is the set of equivalence classes of measurable functions that are essentially bounded: we will say that $f: \Omega \rightarrow \R$ belongs to $L^\infty(\Omega)$ if there exists $y \in \R$ such that $\leb(\{x\in\Omega:|f(x)| > y\}) = 0$.
	In this case,
	$$\norm{f}_\infty = \inf\{y \in \R : \leb(\{x\in\Omega:f(x) > y\}) = 0 \}.$$
	If $1 < p < \infty$, we recall that $p'$ is defined as the unique solution to the equation
	$$
	\frac{1}{p}+\frac{1}{p'}=1,
	$$
	while $1' = \infty$ and $\infty' = 1$.
	
	\item It is well-established that the distributional derivative allows to define a notion of weak derivative for $L^p$ functions.
	$L^2$ functions whose weak derivatives up to order $p < \infty$ are still $L^2$ functions are of a particular relevance in the study of partial differential equations.
	%let $\Omega \subseteq \R^k$ be an open set or the closure of an open set, and
	For $p \in \N$, $p \geq 1$, the space $H^p(\Omega)$ is defined as
	$$H^p(\Omega) = \{ f \in L^2(\Omega) : D^\alpha f \in L^2(\Omega) \text{ for every } \alpha \text{ with }|\alpha|\leq p \}.$$
	We also consider the following norm over the space $H^p(\Omega)$:
	$$
	\norm{f}_{H^p} = \sum_{|\alpha|\leq p} \norm{D^\alpha f}_2,
	$$
	and we will call it the $H^p$ norm.
	Recall also that $H_0^p(\Omega) \subset H^p(\Omega)$ is defined as the closure of $\tests(\Omega)$ in $H^p(\Omega)$ with respect to the $H^p$ norm.
	For further properties of the weak derivative and of the spaces $H^p(\Omega)$ and $H_0^p(\Omega)$, we refer to \cite{strichartz, tartar}.
	
	\item $\rad(\R) = \{ \nu : \nu \text{ is a Radon measure over } \R \text{ satisfying } |\nu|(\R)<+\infty \}$.
	\item $\prob(\R) = \{ \nu \in \rad(\R) : \nu \text{ is a probability measure}\}$.
	
	\item Following \cite{balder, ball, webbym} and others, measurable functions $\nu : \Omega \rightarrow \prob(\R)$ will be called Young measures.
	Measurable functions $\nu : \Omega \rightarrow \rad(\R)$ will be called parametrized measures, even though in the literature the term parametrized measure is used as a synonym for Young measure.
	If $\nu$ is a parametrized measure and if $x \in \Omega$, we will write $\nu_x$ instead of $\nu(x)$.
\end{itemize}

Throughout the paper, we will work with a $|\test'(\Omega)|$-saturated nonstandard model $\hR$, and we will assume familiarity with the basics of nonstandard analysis.
For an introduction on the subject, we refer for instance to Goldblatt \cite{go}, but see also \cite{nsa theory apps, da, keisler, nsa working math, nsa robinson}.
The definitions introduced so far are extended by transfer, as usual, so for instance $\hR^k$ is the set of $k$-uples of hyperreal numbers, and $\{e_1, \ldots, e_k\}$ is a basis of $\hR^k$.

We will denote by $\fin$ the set of finite numbers in $\hR$, i.e.\ $\fin = \{ x \in \hR : x$ is finite$\}$.
For any $x, y \in \ns{\R}$ we will write $x \sim y$ to denote that $x-y$ is infinitesimal, % (i.e.\ $x$ is infinitesimally close to $y$).
we will say that $x$ is finite if there exists a standard $M \in \R$ satisfying $|x| < M$, and we will say that $x$ is infinite whenever $x$ is not finite.
The notion of finiteness can be extended componentwise to elements of $\ns{\R^k}$ whenever $k\in\N$: we will say that $x \in \ns{\R^k}$ is finite iff all of its components are finite, and we define $\sh{x} = (\sh{x_1}, \sh{x_2}, \ldots, \sh{x_k}) \in \R^k$.
Similarly, if $x, y \in \ns{\R^k}$, we will write $x \sim y$ if $|x-y|\sim0$ (notice that this is equivalent to $|x-y|_\infty \sim 0$).

If $k$ is finite, for all finite $x \in \hR^k$ we will denote by $\sh{x}\in\R^k$ the standard part of $x$, i.e.\ the unique vector in $\R^k$ closest to $x$.
Similarly, for any  $A \subseteq \hR^k$, $\sh{X}$ will denote the set of the standard parts of the finite elements of $X$.

The space of grid functions is defined as the space of functions whose domain is a uniform hyperfinite grid.

\begin{definition}\label{def grid functions}
	Let $N_0\in\ns{\N}$ be an infinite hypernatural number.
	Set $N = N_0!$ and $\varepsilon = 1/N$, and define %$\Lambda^k = \{ (n_1\varepsilon, \ldots, n_k\varepsilon) : n_1, \ldots, n_k \in [-N^2, N^2] \subseteq\,\! ^\ast\Z\}$ for all $k \in \N$.
	$$\Lambda = \{ n\varepsilon : n \in  [-N^2, N^2] \cap \ns{\Z}\}.$$
	%\item We set $\dt = \varepsilon^2/\alpha$ for some finite $\alpha \in \ns{\R}$ and we define $T = \{n\dt : n \in \ns{\Z} \mathrm{\ and\ } 0 \leq n \}$.
	%\item Against the general rule introduced above, we will denote a generic point $x \in \ns{\R}^{N+1}$ by $x = (x_0, x_1, \ldots, x_N)$.
	We will say that an internal function $f : \Lambda^k \rightarrow \hR$ is a grid function and, if $A \subseteq \Lambda^k$ is internal, we denote by $\grid{A}$ the space of grid functions defined over $A$:
	$\grid{A} = \mathbf{Intl}(\hR^{A}) = \{ f : A \rightarrow \hR \text{ and } f \text{ is internal}\}.$
\end{definition}

%If $A\subseteq \Lambda^k$, we will denote by $\loeb{A}$ the Loeb measure of $A$ associated to the Lebesgue measure over $\R^k$.
%For the properties of the Loeb measure, we refer to \cite{nsa theory apps, nsa working math}.
%If $A$ is internal, then $A$ is a hyperfinite set and it has a well-defined hyperfinite cardinality, that we will denote by $|A|$.
%In this case, the space of grid functions over $A$ can be identified with a subspace of $\hR^{|A|}$.

\subsection{Some elements of nonstandard topology}\label{ns topology}

In the next definition, we will give a canonical extension of subsets of the standard euclidean space $\R^k$ to internal subsets of the grid $\Lambda^k$.

\begin{definition}
	For any $A \subseteq \R^k$, we define $A_\Lambda = \ns{A} \cap \Lambda^k$.
	Notice that $A_\Lambda$ is an internal subset of $\Lambda^k$, and in particular it is hyperfinite.
	%If $f:A\rightarrow \R$, we will denote by $f_\Lambda$ the function $f_\Lambda \in \grid{A_\Lambda}$ defined by $f_\Lambda(x) = \ns{f}(x)$.
\end{definition}

In general, %however, 
we expect that for a generic set $A \subseteq \hR^k$, $\sh{A_\Lambda} \not = \overline{A}$.
For instance, if $A \cap \Q^k = \emptyset$, then $A_\Lambda = \sh{A_\Lambda} = \emptyset$.
In this section, we will prove that if $A$ is an open set, then indeed $A_\Lambda$ is a faithful extension of $A$, in the sense that $\sh{A_\Lambda} = \sh{\overline{A}_\Lambda} = \overline{A}$.
Moreover, there is a nice characterization of the boundary of $A_\Lambda$ which is projected to the boundary of $A$ via the standard part map.

In order to prove these results, we need to show that for an open set $A$, $\mu(x) \cap \ns{A} \not = \emptyset$ is equivalent to $\mu(x) \cap A_\Lambda \not = \emptyset$ for all $x\in \overline{A}$.
%The proof of this result requires a preliminary Lemma.

\begin{lemma}\label{aperti}
	If $A \subseteq \R^k$ is an open set, then for all $x \in \overline{A}$ it holds
	\begin{equation}\label{closed3}
	\mu(x) \cap \ns{A} \not= \emptyset \Longleftrightarrow \mu(x) \cap A_\Lambda \not= \emptyset.
	\end{equation}
\end{lemma}
\begin{proof}
	Let $x \in \overline{A}$.
	The hypothesis $N=N_0!$ for an infinite $N_0 \in \ns{\N}$ ensures that for all $p \in \Q^k$, $p \in\Lambda^k$.
	As a consequence, for all $n \in \N$ there exists $p \in A_\Lambda$ with $|x-p|<1/n$.
	By overspill, for some infinite $M \in \ns{\N}$ there exists $p \in  A_\Lambda$ that satisfies $|x-p|<1/M$.
\end{proof}

We want to define a boundary for the set $A_\Lambda$ that is coherent with the usual notion of boundary for $A$.
The idea is to define the $\Lambda$-boundary of $A_\Lambda$ as the set of points of $A_\Lambda$ that are within a step of length $\varepsilon$ from a point of $\ns{A}^c$.

\begin{definition}
	Let $A \subseteq \R^k$.
	We define the $\Lambda$-boundary of $A_\Lambda$ as
	$$
	\partial_\Lambda A_\Lambda = \{ x \in A_\Lambda : \exists y \in \ns{A}^c \text{ satisfying } |x-y|_\infty\leq \varepsilon \}.
	$$
	%and, for any multi-index $\alpha$, the $\alpha$-boundary of $A$ as
	%$$\partial^\alpha_\Lambda A = \{ x \in A : x+\alpha\varepsilon \in \partial_\Lambda A \}.$$
\end{definition}

This definition is coherent with the usual boundary of an open set.

\begin{proposition}\label{topologia bella}
	Let $A \subseteq \R^k$ be an open set.
	Then $\sh{A_\Lambda} = \overline{A}$ and $\sh{ (\partial_\Lambda A_\Lambda)} = \partial A$.
\end{proposition}
\begin{proof}
	The equality $\sh{A_\Lambda} = \overline{A}$ is a consequence of Lemma \ref{aperti}.
	
	Recall the nonstandard characterization of the boundary of $A$: $x \in \partial A$ if and only if there exists $y \in \ns{A}$, $x \not = y$, and $z \in \ns{A}^c$ with $x \sim y \sim z$.
	This is sufficient to conclude that $\partial A \supseteq \sh{ (\partial_\Lambda A_\Lambda)}$.
	
	To prove that the other inclusion holds, we only need to show that if $x \in \partial A$, then there exists $y \in \partial_\Lambda A_\Lambda$ with $y \sim x$.
	Let $x \in \partial A$: since $A_\Lambda$ is a hyperfinite set, we can pick $y \in A_\Lambda$ satisfying $$|\ns{x}-y|_\infty = \min_{z \in A_\Lambda}\{|\ns{x}-z|_\infty\}.$$
	Our choice of $y$ and the hypothesis that $x \in \partial A$ ensure that $y \not = \ns{x}$ and $|\ns{x}-y|_\infty \sim 0$.
	We claim that $y \in \partial_\Lambda A_\Lambda$.
	In fact, suppose towards a contradiction that $y \not\in \partial_\Lambda A_\Lambda$: in this case, for all $z \in \ns{A}^c$, $|y-z|_\infty>\varepsilon$ and, in particular, $|\ns{x}-y|_\infty > \varepsilon$. % and, in particular, $y+\sum_{i = 1}^k \varepsilon e_i d_i \in A_\Lambda$ for all choices of $d_i \in \{0,1\}$.
	Let $\ns{x}-y = \sum_{i = 1}^k a_i e_i$, let $I = \{i \leq k : |a_i| = |\ns{x}-y|_\infty\}$, and define
	$$\tilde{y} = y+\sum_{i \in I} \frac{a_i}{|a_i|}\varepsilon e_i.$$
	Since $|\tilde{y}-y|_\infty = \varepsilon$ and since $y \not\in \partial_\Lambda A_\Lambda$, then $\tilde{y} \in A_\Lambda$.
	Moreover,
	$$|\ns{x}-\tilde{y}|_\infty = \max_{i\not\in I}\{|\ns{x}-y|-\varepsilon, |a_i|\} < |x-y|_\infty,$$
	contradicting $|\ns{x}-y|_\infty = \min_{z \in A_\Lambda}\{|\ns{x}-z|_\infty\}$.
\end{proof}

Let $\Omega \subseteq \R^k$ be an open set.
By Proposition \ref{topologia bella}, this hypothesis is sufficient to ensure the equalities $\sh{\dom} = \sh{\overline{\Omega}_\Lambda} = \overline{\Omega}$ and $\sh{(\partial_\Lambda \dom)} = \sh{(\partial_\Lambda \overline{\Omega}_\Lambda)} = \partial \Omega$.

\subsection{Derivatives and integrals of grid functions}

Since grid functions are defined on a discrete set, there is no notion of derivative for grid functions.
However, in nonstandard analysis it is fairly usual to replace the derivative by a finite difference operator with an infinitesimal step.

\begin{definition}[Grid derivative]\label{fd}
	For an internal grid function $f \in \grid{\dom}$, we define the $i$-th forward finite difference of step $\varepsilon$ as
	\begin{equation*}
	\D_i f(x)= \D_i^+ f(x)= \frac{f(x+\varepsilon e_i)-f(x)}{\varepsilon }
	\end{equation*}
	and the $i$-th backward finite difference of step $\varepsilon$ as
	$$
	\D_i^- f(x) = \frac{f(x)-f(x-\varepsilon e_i)}{\varepsilon }.
	$$
	%For $f : \Lambda^k \rightarrow \hR$, the functions $\D_{i} f(x)$ with $0 \leq i \leq k$ are defined in the expected way.
	If $n \in \ns{\N}$, $\D^n_i$ is recursively defined as $\D_i(\D_i^{n-1})$ and, if $\alpha$ is a multi-index, then $\D^\alpha$ is defined as expected:
	$$
	\D^\alpha f = \D_1^{\alpha_1} \D_2^{\alpha_2} \ldots \D_n^{\alpha_k} f.
	$$
	These definitions can be extended to $\D^-$ by replacing every occurrence of $\D$ with $\D^-$.
	%Most of the time, we will write $\D_\alpha$ for $\D_\alpha^+$.
\end{definition}
%By internal induction, it is possible to prove that for all $M \in \ns{\N}$ it holds
%\begin{equation}\label{derivata discreta ordine M}
%\D^M f(x) = \sum_{i = 0}^M (-1)^{M-i} \binom{M}{i} f(x+i\varepsilon).
%\end{equation}
For further details about the properties of the finite difference operators, we refer to Hanqiao, St.\ Mary and Wattenberg \cite{watt}, to Keisler \cite{keisler} and to van den Berg \cite{imme2, imme1}.

\begin{remark}
	Notice that if $f \in \grid{\dom}$ and if $\alpha$ is a standard multi-index, then $\D^{\alpha} f$ is not defined on all of $\dom$.
	However, if we let
	$$
	\dom^\alpha = \{ x \in \dom : \D^{\alpha} f \text{ is defined} \}
	= \{ x \in \dom : x+\alpha\varepsilon\in\dom \}
	$$
	then we have $\sh{\dom^\alpha} = \sh{\dom} = \overline{\Omega}$, since for every $x \in \dom^\alpha$ we have $x+\alpha\varepsilon\in\dom$ and $x+\alpha\varepsilon \sim x$ by the standardness of $\alpha$.
	
	In a similar way, if we define %The $\Lambda$-boundary of $\dom$ can then be represented in $\dom^\alpha$ by the set
	$$
	\partial_\Lambda^\alpha \dom = \{ x \in \dom : x+\alpha\varepsilon\in\partial_\Lambda\dom \},
	$$
	then, from the relation $x+\alpha\varepsilon \sim x$ and from Proposition \ref{topologia bella}, we deduce that it holds also the equality $\sh{\partial_\Lambda^\alpha \dom} = \sh{\partial_\Lambda \dom} = \partial\Omega$.
	In section \ref{linear pdes}, we will use this result in order to show show how Dirichlet boundary conditions can be expressed in the sense of grid functions.%, and in section \ref{ns models} we will discuss the grid function formulation of a PDE with Neumann boundary conditions.
	
	Since $\dom^\alpha$ is a faithful extension of $\Omega$ in the sense of proposition \ref{topologia bella}, we will often abuse notation and write $\D^{\alpha} f \in \grid{\dom}$ instead of the correct $\D^\alpha f \in \grid{\dom^\alpha}$.
\end{remark}

In the setting of grid functions, integrals are replaced by hyperfinite sums.

\begin{definition}[Grid integral and inner product]\label{def inner}
	Let $f,g : \ns{\Omega} \rightarrow \hR$ and let $A \subseteq \dom \subseteq \Lambda^k$ be an internal set.
	We define
	$$
	\int_{A} f(x) d\Lambda^k = \varepsilon^k \cdot \sum_{x \in A} f(x)
	$$
	and
	$$
	%\begin{array}{rcl}
	\langle f, g \rangle
	=  \displaystyle \int_{\Lambda^k} f(x) g(x) d\Lambda^k %\\ \\
	=  \displaystyle \varepsilon^k \cdot \sum_{x \in \Lambda^k} f(x)g(x),
	%\end{array}
	$$
	with the convention that, if $x \not \in \ns{\Omega}$, $f(x) = g(x) = 0$.
\end{definition}

A simple calculation shows that the fundamental theorem of calculus holds.
In particular, for all $f:\grid{\Lambda}\rightarrow \hR$ and for all $a, b \in \Lambda$, $b < N$, we have
$$\varepsilon \sum_{x=a}^{b} \D f(x) = f(b+\varepsilon)-f(a) \text{ and } \D \left( \varepsilon \sum_{x=a}^b f(x) \right) = f(b+\varepsilon).$$
The next Lemma is a well-known compatibility result between the grid integral and the Riemann integral of continuous functions.

\begin{lemma}\label{equivalenza integrali}
	%Let $\ni{a}{b}$ be a proper near-interval.
	Let $A \subset \R^k$ be a compact set.
	If $f \in C^0(A)$, then $$\int_{A_{\Lambda}} \ns{f}(x) d\Lambda^k \sim \int_{A} f(x) dx.$$	
\end{lemma}
\begin{proof}
	See for instance Section 1.11 of \cite{nsa working math}.
\end{proof}
%Ricorda qualche criterio di indistinguibilità con l'integrale di Lebesgue (S-integrabilità eccetera). Vedi Cutland e Rodenhausen.

In order to introduce the grid functions that correspond to real distributions, we will use a notion of duality induced by the inner product \ref{def inner}.

\begin{definition}
	For any $V \subseteq \grid{\dom}$, we define
	$$V' = \{ f \in \grid{\dom} : \langle g, f \rangle \text{ is finite for all } g \in V\}.$$
\end{definition}

\begin{lemma}
	For any $V \subseteq \grid{\dom}$, $V'$ with pointwise sum and product is a module over $\fin$.
	Moreover, $V'/\equiv$ inherits a structure of real vector space from $V'$.
\end{lemma}

Notice that, contrary to what happened for the space $S^0(\dom)$, $V'$ is not an algebra, since in general the hypothesis $f, g \in V'$ is not sufficient to ensure that $fg \in V'$.

\subsection{$S^\alpha$ functions and $C^\alpha$ functions}\label{S-continuous}

We will now define functions of class $S^\alpha$, that will be the grid functions counterpart of functions of class $C^\alpha$.
This definition is grounded upon the well-known notion of S-continuity, as S-continuity has been widely used as a bridge between discrete functions of nonstandard analysis and standard continuous functions.

\begin{definition}
	We will say that $x \in \dom$ is nearstandard in $\Omega$ iff there exists $y \in \Omega$ such that $x \sim y$.
\end{definition}

\begin{definition}
	We say that a function $f \in \grid{\dom}$ is of class $S^0$ iff
	$f(x)$ is finite for some nearstandard $x \in \dom$ and
	for every nearstandard $x, y \in \dom$, $x \sim y$ implies $f(x) \sim f(y)$.
	
	We also define functions of class $S^\alpha$ for every multi-index $\alpha$:
	\begin{itemize}
		%\item $f$ is of class $S^0(\dom)$ if $f$ is S-continuous on $\dom$;
		\item $f$ is of class $S^\alpha(\dom)$ if $\D^\alpha f \in S^0(\dom)$;
		\item $f$ is of class $S^{\infty}(\dom)$ if $\D^\alpha f \in S^0(\dom)$ for any standard multi-index $\alpha$.
	\end{itemize}
\end{definition}

Notice that if $f \in S^\alpha(\dom)$ for some standard multi-index $\alpha$, then $f(x)$ is finite at all nearstandard $x \in \dom$.

In the study of $S$-continuous functions, we find it useful to introduce the following equivalence relation.

\begin{definition}\label{equiv1}
	Let $f, g \in \grid{\dom}$.
	We say that $f \equiv_S g$ iff $(f-g)(x)\sim 0$ for all nearstandard $x \in \dom$.
	From the properties of $\sim$, it can be proved that $\equiv_S$ is an equivalence relation.
	We will denote by $\pi_S$ the projection from $\grid{\dom}$
	to the quotient space $\grid{\dom} / \equiv_S$,
	and will denote by $[f]_S$ the equivalence class of $f$ with respect to $\equiv_S$.
\end{definition}

The rest of this section is devoted to the proof that the quotient $S^\alpha(\dom) / \equiv_S$ is real algebra isomorphic to the algebra of $C^\alpha$ functions over $\Omega$.
This result is a reformulation in the language of grid functions of some results by van den Berg \cite{imme2} and by Wattenberg, Hanqiao, and St.\ Mary \cite{watt}.

\begin{lemma}\label{lemma piccolo}
	%Let $\dom$ be a near-interval.
	For every standard multi-index $\alpha$, $S^\alpha(\dom)$ with pointwise sum and product is an algebra over $\fin$, and $S^\alpha(\dom) / \equiv_S$ inherits a structure of real algebra from $S^\alpha(\dom)$.
\end{lemma}
\begin{proof}
	The only non-trivial assertion that needs to be verified is closure of $S^\alpha(\dom)$ with respect to pointwise product.
	This property is a consequence of Proposition 2.6 of \cite{imme2}.
\end{proof}

\begin{theorem}\label{teorema isomorfismo}
	%Let $I$ be a near-interval.
	$S^0(\dom)/\equiv_S$ is a real algebra isomorphic to $C^0(\Omega)$.
	The isomorphism is given by $i[f]_S=\sh{f}$.
	The inverse of $i$ is the function $i^{-1}(f) = [\ns{f}_{|\Lambda}]_S$.
\end{theorem}
\begin{proof}
	%It is clear that $S^0(\dom)/\equiv_S$ is a real algebra.
	If $f \in S^0(\dom)$, then it is well-known that $\sh{f}$ is a well-defined function and that $\sh f \in C^0(\Omega)$.
	Surjectivity of $\sh{}$ is a consequence of Lemma II.6 of \cite{watt}.
	Since
	$$
	\ker(\sh{}) = \left\{ f \in S^\alpha(\dom) : f(x) \sim 0 \text{ for all finite } x \in \dom \right\} = [0]_S,
	$$
	we deduce that $i$ is injective and surjective.
	Since $\sh{(x+y)} = \sh{x}+\sh{y}$ and $\sh{(xy)} = \sh{x}\sh{y}$ for all $x, y \in \fin$, $i_\alpha$ is an isomorphism of real algebras. 
\end{proof}

We will now show that, for grid functions of class $S^\alpha$, the finite difference operators $\D_i^+$ and $\D_i^-$ assume the role of the usual partial derivative for $C^\alpha$ functions.
In particular, these finite difference operators can be seen as generalized derivatives.

\begin{theorem}\label{teorema equivalenza derivate}
	For all $1 \leq i \leq k$ and for all standard multi-indices $\alpha$ with $\alpha_i\geq1$, the diagrams
	$$
	\begin{array}{ccc}
	\begin{array}{ccc}
	S^{\alpha}(\dom) & \stackrel{\D^+_i}{\longrightarrow} & S^{\alpha-e_i}(\dom) \\
	i \circ \pi_S \downarrow & & \downarrow i \circ \pi_S\\
	C^\alpha(\Omega) & \stackrel{D_i}{\longrightarrow} & C^{\alpha-e_i}(\Omega)
	%	i \updownarrow & & \updownarrow i \\
	%	C^n([\sh{a}, \sh{b}]) & \stackrel{'}{\longrightarrow} & C^{n-1}([\sh{a}, \sh{b}])
	\end{array}
	&
	\text{ and }
	&
	\begin{array}{ccc}
	S^\alpha(\dom) & \stackrel{\D^-_i}{\longrightarrow} & S^{\alpha-e_i}(\dom) \\
	i \circ \pi_S \downarrow & & \downarrow i \circ \pi_S\\
	C^\alpha(\Omega) & \stackrel{D_i}{\longrightarrow} & C^{\alpha-e_i}(\Omega)
	\end{array}
	\end{array}
	$$
	commute.	
\end{theorem}
\begin{proof}
	By Theorem \ref{teorema isomorfismo}, if $f \in S^\alpha(\dom)\subseteq S^0(\dom)$ then $(i_\alpha \circ \pi_S)(f) = \sh{f}$ and, by Lemma II.7 of \cite{watt}, $\sh{(\D^{\pm}_i f)} = D_i \sh{f}$.
\end{proof}

By Theorem \ref{teorema equivalenza derivate}, the isomorphism $i$ defined in Theorem \ref{teorema isomorfismo} induces an isomorphism between $S^\alpha(\dom) / \equiv_S$ and $C^\alpha(\Omega)$ as real algebras.

\begin{corollary}\label{corollario isomorfismo}
	For any multi-index $\alpha$, the isomorphism $i$ restricted to $S^\alpha(\dom) / \equiv_S$ induces an isomorphism between $S^\alpha(\dom) / \equiv_S$ and $C^\alpha(\Omega)$ as real algebras.
\end{corollary}

Thanks to this isomorphism, if $f \in S^\alpha(\dom)$, we can identify the equivalence class $[f]_S$ with the standard function $\sh{f} \in C^\alpha(\Omega)$.

%\begin{corollary}
%Let $I$ be a near-interval.
%	If $f \in S^n(\dom)$, then for all $1 \leq i \leq k$ and for all $k \leq n$, $[(\D^{\pm}_i)^{k} f] = (\sh{f})^k$.
%\end{corollary}

\section{Grid functions as generalized distributions\label{distri}}

In this section, we will study the relations between the space of grid functions and the space of distributions.
In particular, we will prove that the space of grid functions can be seen as generalization of the space of distributions, and that the operators $\D^+$ and $\D^-$ coherently extend the distributional derivative to the space of grid functions.

In order to prove the above results, we start by defining a projection from an external $\fin$-submodule of $\grid{\dom}$ to the space of distributions.
This projection is defined by duality with an external $\fin$-algebra of grid functions that is a counterpart to the space of test functions.%, and by proving that the quotient of this space by the equivalence relation $\equiv_S$ is indeed isomorphic to the usual space of real test functions.

\begin{definition}[Algebra of test functions]
	%Let $I$ be a near-interval.
	We define the algebra of test functions over $\dom$ as follows:
	$$
	\test(\dom) =
	\left\{ f \in S^{\infty}(\dom) : \sh{\supp f} \subset \subset \Omega \right\}.
	$$
\end{definition}

The above definition provides a nonstandard counterpart of the usual space of smooth functions with compact support.

\begin{lemma}\label{lemma test}
	The isomorphism $i$ defined in Theorem \ref{teorema isomorfismo} induces an isomorphism between the real algebras
	$\test(\dom) / \equiv_S$ and $\tests(\Omega)$.
	The isomorphism preserves integrals, i.e.\ for all $\varphi \in \test(\dom)$, it holds the equality
	\begin{equation}\label{uguaglianza integrali}
	\sh{ \int_{\dom} \varphi d\Lambda^k } = \int_{\Omega} i[\varphi]_S dx.
	\end{equation}
	Moreover, if $\varphi \in \tests(\Omega)$, then $\ns{\varphi}_{|\Lambda} \in \test(\dom)$, so that $i^{-1}(\varphi) = \left[\ns{\varphi}_{|\Lambda}\right]_S \cap \test(\dom)$.
\end{lemma}
\begin{proof}
	From Theorem \ref{teorema isomorfismo}, from Theorem \ref{teorema equivalenza derivate} and from the definition of $\test(\dom)$, we can conclude that the hypothesis $\varphi \in \test(\dom)$ ensures that $i[\varphi] \in \tests(\Omega)$.
	Since $\test(\dom) \subset S^0(\dom)$, injectivity of $i$ is a consequence of Theorem \ref{teorema isomorfismo}.
	
	Similarly, surjectivity of $i$ can be deduced from Theorem \ref{teorema isomorfismo} and from Theorem \ref{teorema equivalenza derivate}.
	In fact, suppose towards a contradiction that there exists $\psi \in \tests(\Omega)$ such that $\psi \not = i[\varphi]$ for all $\varphi \in \test(\Omega)$.
	Since $\psi \in C^0(\Omega)$, Theorem \ref{teorema isomorfismo} ensures that there exists $\phi \in S^0(\dom)$ with $i[\phi]=\psi$.
	If $\phi \not \in S^\infty(\dom)$, then for some standard multi-index $\alpha$, $\D^\alpha \phi \not \in S^0(\dom)$, contradicting Theorem \ref{teorema equivalenza derivate}.
	As a consequence, $i$ is an isomorphism between $\test(\dom) / \equiv_S$ and $\tests(\Omega)$.
	
	Equality \ref{uguaglianza integrali} is a consequence of the hypothesis $\sh{\supp \varphi} \subset \subset \Omega$ and of Lemma \ref{equivalenza integrali}.
	
	Now let $\varphi \in \tests(\Omega)$: by Theorem \ref{teorema isomorfismo} and by Theorem \ref{teorema equivalenza derivate}, $\ns{\varphi}_{|\Lambda} \in S^\infty(\dom)$.
	Let $A = \supp \varphi$: since $A$ is the closure of an open set, by Proposition \ref{topologia bella} $\sh{A_\Lambda} = A \subset \subset \Omega$, from which we deduce $\ns{\varphi}_{|\Lambda} \in \test(\dom)$.
	As a consequence, $i^{-1}(\varphi) = \left[\ns{\varphi}_{|\Lambda}\right]_S \cap \test(\dom)$, as we claimed.
\end{proof}

The duality with respect to the space of test functions can be used to define an equivalence relation on the space of grid functions.
This equivalence relation plays the role of a weak equality.

\begin{definition}\label{def equiv}
	Let $f, g \in \grid{\dom}$.
	We say that $f \equiv g$ iff for all $\varphi \in \test(\dom)$ it holds
	$
	\langle f, \varphi \rangle \sim \langle g, \varphi \rangle.
	$
	%We define the space of generalized distributions on $\dom$ as the quotient $\grid{\dom}/\equiv$.
	We will call $\pi$ the projection from $\grid{\dom}$ to the quotient
	$
	\grid{\dom} / \equiv,
	$
	and we will denote by $[f]$ the equivalence class of $f$ with respect to $\equiv$.
\end{definition}

The new equivalence relation $\equiv$ is coarser than $\equiv_S$.

\begin{lemma}
	For all $f, g \in \grid{\dom}$, $f \equiv_S g$ implies $f \equiv g$.
\end{lemma}
\begin{proof}
	We will show that $f \equiv_S g$ implies $\langle f-g, \varphi\rangle \sim 0$ for all $\varphi \in \test(\dom)$: by linearity of the hyperfinite sum, this result is equivalent to $f \equiv g$.
	
	Let $\varphi \in \test(\dom)$, and let $\eta = \max_{x \in \supp \varphi} \{|(f-g)(x)|\}$.
	The hypothesis that $f \equiv_S g$ and the hypothesis that $\sh{\supp \varphi}$ is bounded are sufficient to ensure that $\eta \sim 0$.
	As a consequence, we have the following inequalities
	$$
	\displaystyle
	\left| \langle f-g,\varphi\rangle \right|
	\leq 
	\displaystyle
	\langle |f-g|,|\varphi|\rangle 
	\leq 
	\displaystyle
	|\eta| \int_{\dom} |\varphi(x)| d\Lambda^k \sim 0,
	$$
	%from which we deduce that $\langle f-g, \varphi\rangle \sim 0$.
	that are sufficient to conclude the proof.
\end{proof}

We can now define the grid distributions as the the dual of $\grid{\dom}$ with respect to the inner product introduced in Definition \ref{def inner}.

\begin{definition}\label{def bounded grid functions}
	The $\fin$-module
	$$
	\test'(\dom)
	=
	\left\{ f\in \grid{\dom}\ |\ \langle f, \varphi \rangle \text{ is finite for all } \varphi\in\test(\dom)\right\}
	$$
	is called the module of grid distributions.
\end{definition}

The rest of this section is devoted to the proof that the quotient $\test'(\dom) / \equiv$ is real vector space isomorphic to the space of distributions $\tests'(\Omega)$.
The following characterization of grid distributions will be used in the proof of this isomorphism.

\begin{lemma}\label{character}
	The following are equivalent:
	\begin{enumerate}
		\item $f \in \B$;
		\item $\langle f, \varphi \rangle \sim 0$ for all $\varphi \in \test(\dom)$ satisfying $\varphi(x) \sim 0$ for all $x \in \dom$;
		\item $\langle f, \ns{\varphi} \rangle \in \fin$ for all $\varphi \in \tests(\Omega)$.
	\end{enumerate}
\end{lemma}
\begin{proof}
	(1) implies (2), by contrapositive.
	Suppose that $\langle f, \varphi \rangle \not \sim 0$ for some $\varphi \in \test(\dom)$ with $\varphi(x) \sim 0$ for all $x \in \dom$.
	If $\varphi(x) \geq 0$ for all $x\in\dom$, take some $\psi \in \test(\dom)$ with $\psi(x) \geq n\varphi(x)$ for all $x \in \dom$ and for all $n \in \N$.
	From the inequality $|\langle f, \psi \rangle| \geq n |\langle f, \varphi \rangle|$ for all $n \in \N$, we deduce that $\langle f, \psi \rangle$ is infinite, i.e.\ that $f \not \in \B$.
	The remaining cases, namely if $\varphi(x) \leq 0$ for all $x \in \dom$ or if there exists $x, y \in \dom$ with $\varphi(x)\varphi(y)<0$, can be dealt in a similar way, thanks to the linearity of the hyperfinite sum. 
	
	(2) implies (1), by contrapositive.
	Suppose that $\langle f, \varphi \rangle = M$ is infinite for some $\varphi \in \test(\dom)$.
	Since $\varphi/M \in \test(\dom)$ and $\varphi/M(x) \sim 0$ for all $x \in \dom$, we deduce that (2) does not hold.
	
	It is clear that (1) implies (3).
	
	(3) implies (1), by contradiction.
	Suppose that $\langle f, \ns{\varphi} \rangle \in \fin$ for all $\varphi \in \tests(\Omega)$, but that $f \not \in \B$.
	Since (1) and (2) are equivalent, there exists $\psi \in \test(\dom)$ with $\psi(x) \sim 0$ for all $x \in \dom$ such that $\langle f, \psi\rangle \not \sim 0$.
	By reasoning as in the first part of the proof, we deduce that there exists $\phi \in \tests(\Omega)$ with $\langle f, \ns{\phi} \rangle \not\in \fin$, a contradiction.
\end{proof}

From the above Lemma, we deduce that the action of a grid distribution over the space of test functions is continuous.

\begin{corollary}[Continuity]\label{cor continuity}
	If $\varphi, \psi \in \test{(\dom)}$ and $\varphi \equiv_S \psi$, then $\langle f, \varphi\rangle \sim \langle f, \psi\rangle$ for all $f \in \B$.
\end{corollary}
\begin{proof}
	The hypotheses $\varphi, \psi \in \test{(\dom)}$ and $\varphi \equiv_S \psi$ imply $\varphi - \psi \in \test(\dom)$ and $(\varphi - \psi)(x) \sim 0$ for all $x \in \dom$.
	Then, by Lemma \ref{character}, we have $\langle f, \varphi - \psi\rangle \sim 0$
	for all $f \in \B$, as we wanted.
\end{proof}

\begin{comment}

\begin{corollary}\label{corollario utile}
For all $\varphi, \psi \in \test(\dom)$ and for all $f, g \in \B(\dom)$, if $\varphi \equiv_S \psi$ and if $f \equiv g$, then
$$
\langle f, \varphi \rangle \sim \langle g, \psi \rangle.
$$
\end{corollary}
\begin{proof}
By definition of $\equiv$, $\langle f, \eta \rangle \sim \langle g, \eta \rangle$ for all $\eta \in \test(I)$.
By Corollary \ref{corollario utile}, $\langle h, \varphi \rangle \sim \langle h, \psi \rangle$ for all $h \in \B(I)$.
We deduce that
$$
\langle f, \varphi \rangle \sim \langle g, \varphi \rangle \sim \langle g, \psi \rangle
$$
as we wanted.
\end{proof}
\end{comment}

We are now ready to prove that $\B/\equiv$ is isomorphic to the space of distributions over $\Omega$.

\begin{theorem}
	\label{bello}
	The function $\Phi:(\B/\equiv) \rightarrow \tests'(\Omega)$ defined by
	$$
	\ldual \Phi([f]), \varphi \rdual
	=
	\sh{\langle f, \ns{\varphi} \rangle}
	$$
	is an isomorphism of real vector spaces.% between $\B/\equiv$ and $\tests'(\Omega)$.
\end{theorem}
\begin{proof}
	At first, we will show that the definition of $\Phi$ does not depend upon the choice of the representative for $[f]$.
	%Since $f \in \test'(\dom)$, $\sh{\langle f, \ns{\varphi}_{|\Lambda} \rangle}$ is finite.
	Let $g, h \in [f]$: then, by definition of $\equiv$, $\sh{\langle g, \varphi \rangle} = \sh{\langle h, \varphi \rangle}$ for all $\varphi\in\test(\dom)$.
	By Lemma \ref{lemma test}, for all $\varphi \in \tests'(\Omega)$, $\ns{\varphi}_{|\Lambda} \in \test(\dom)$, so that if $g,h \in [f]$,  then $\sh{\langle g, \ns{\varphi} \rangle} = \sh{\langle h, \ns{\varphi} \rangle}$
	so that the definition of $\Phi$ is independent on the choice of the representative for $[f]$.
	
	Lemma \ref{cor continuity} ensures that for all $[f]\in \B/\equiv$, $\Phi([f]) \in \test'(\Omega)$, and in particular that $\Phi([f])$ is continuous.
	
	We will prove by contradiction that $\Phi$ is injective.
	Suppose that $\langle \Phi([f]),\varphi\rangle = 0$ for all $\varphi \in \tests(\Omega)$
	and that $[f] \not = [0]$.
	The latter hypothesis implies that there exists $\psi \in \test(\dom)$ such that $\langle f, \psi \rangle \not \sim 0$.
	But, since $\ns{(\sh{\psi})}_{|\Lambda} \equiv_S \psi$, by Corollary \ref{cor continuity} we deduce
	$$
	\ldual \Phi([f]),\sh{\psi}\rdual
	=
	\sh{\langle f, \ns{(\sh{\psi})} \rangle}
	=
	\sh{
		\langle f, \psi \rangle}
	\not = 0,
	$$
	contradicting the hypothesis $\langle \Phi([f]),\varphi\rangle = 0$ for all $\varphi \in \tests(\Omega)$.
	As a consequence, $\Phi$ is injective.
	
	Surjectivity of $\Phi$ is a consequence of Theorem 1 of \cite{moto} and of Lemma \ref{character}.
\end{proof}

Thanks to the previous theorem, from now on we will identify the equivalence class $[f]$ with the distribution $\Phi([f])$.
Notice that if $f \in S^0(\dom)$, this identification is coherent with $[f]_S$.

\begin{corollary}\label{corollario teorema equivalenza}
	If $f \in S^0(\dom)$, then $[f] = [f]_S = \sh{f}$.
\end{corollary}
\begin{proof}
	Since $f$ is S-continuous, by Lemma \ref{equivalenza integrali} and by Lemma \ref{lemma test} we have the equality
	$$
	\int_{\Omega} \sh{f} \varphi dx = \sh{\langle f, \ns{\varphi} \rangle}
	$$
	for all $\varphi \in \tests(\Omega)$, and this is sufficient to deduce the thesis.
\end{proof}

\begin{remark}\label{remark vector valued}
	If $k\in\N$, define
	$$
	\test'(\Omega, \hR^k)= \left\{ f: \dom \rightarrow \R^k : f_i \in \test'(\Omega) \text{ for all } 1 \leq i \leq k \right\}.
	$$
	If $f \in \test'(\Omega, \hR^k)$, then we can define a functional $[f]$ over the dual of the space of vector-valued test functions
	$$
	\tests(\Omega, \R^k) = \left\{ \varphi: \Omega \rightarrow \R^k : \varphi_i \in \tests(\Omega) \text{ for all } 1 \leq i \leq k \right\}
	$$
	by posing
	$
	\langle [f], \varphi \rangle_{\tests(\Omega, \R^k)} = \sum_{i = 1}^k\sh{\langle f_i, \ns{\varphi_i}\rangle}
	$
	for all $\varphi \in \tests(\Omega, \R^k)$.
	From Theorem \ref{bello}, we deduce that the quotient of the $\fin$-module
	$
	\test'(\Omega, \hR^k)%= \left\{ f: \dom \rightarrow \R^k : f_i \in \test'(\Omega) \text{ for all } 1 \leq i \leq k \right\}
	$
	with respect to $\equiv$
	is isomorphic to the real vector space of linear continuous functionals over $\tests(\Omega, \R^k)$.
\end{remark}

\begin{remark}\label{remark estensioni}
	Theorem \ref{bello} can be used to define more general projections of nonstandard functions.
	For instance, if $f \in \ns{C^0}(\hR, \test'(\dom))$, then for all $T \in \R$ $f$ induces a continuous linear functional $[f]$ over the space $C^0([0,T],\tests'(\Omega))$ defined by the formula
	$$
	\int_0^T \ldual [f], \varphi \rdual dt = \sh{\left(\ns{\int}_0^T \langle f(t), \ns{\varphi}(t) \rangle dt \right)}
	$$
	for all $\varphi \in C^0([0,T],\tests'(\Omega))$.
	Moreover, if $f \in \ns{C^1}(\hR, \test'(\dom))$, then $[f]$ allows for a weak derivative with respect to time: for all $T \in \R$, $[f]_t$ is the distribution that satisfies
	$$
	\int_0^T \ldual [f]_t, \varphi \rdual dt = - \sh{\left(\ns{\int}_0^T \langle f(t), \ns{\varphi}(t) \rangle dt \right)}
	$$
	for all $\varphi \in C^1([0,T],\tests'(\Omega))$.
	%\color{red}
	%This result will be useful in Section \ref{section nonli} and in Chapter \ref{chap3}. FINISCI DI COMMENTARE
\end{remark}

\subsection{Discrete derivative and distributional derivative}\label{distributonal derivative}

In Theorem \ref{teorema equivalenza derivate}, we have seen that the finite difference operators $\D_i^+$ and $\D_i^-$ generalize the derivative for smooth functions to the setting of grid functions.
We will now see that it holds a more general result: the operators $\D_i^+$ and $\D_i^-$ generalize also the distributional derivative, in the sense that $[\D_i^{\pm} f] = D_i[f]$ for all $f \in \B$.
For a matter of commodity, we will suppose that $\dom \subseteq \Lambda$:
the generalization to an arbitrary dimension can be deduced from the proof of Theorem \ref{teorema equivalenza derivate2} with an argument relying on Theorem \ref{teorema equivalenza derivate}.

%If $T$ is a distribution in the standard sense, then the distributional derivative $T'$ is defined by $\ldual T', \phi \rdual = -\ldual T, \phi'\rdual$ for all $\phi \in \tests(\dom)$.
Recall the discrete summation by parts formula: for all grid functions $f$ and $g$ and for all $a, b \in \ns{\N}$ with $N^2 \leq a < b < N^2$ it holds the equality
\begin{eqnarray*}
	\sum_{n = a}^{b} (f((n+1)\varepsilon)-f(n\varepsilon))g(n\varepsilon)
	&=& f((b+1)\varepsilon)g((b+1)\varepsilon) - f(a\varepsilon)g(a\varepsilon) +\\
	&&- \sum_{n = a}^{b} f((n+1)\varepsilon) (g((n+1)\varepsilon)-g(n\varepsilon))
\end{eqnarray*}
that, in particular, implies
\begin{equation}\label{party}
\left\langle \D f, \varphi \right\rangle = - \left\langle f(x+\varepsilon), \D \varphi \right\rangle
\end{equation}
for all $f \in \grid{\dom}$ and for all $\varphi \in \test(\dom)$.

Inspired by the above formula, we will now prove that if we shift a grid distribution by an infinitesimal displacement, we still obtain the same grid distribution.

\begin{lemma}\label{shift2}
	Let $f \in \grid{\dom}$.
	Then $f(x) \in \B$ if and only if $f(x+\varepsilon)\in \B$.
	If $f(x) \in \B$ then, $[f(x)] = [f(x+\varepsilon)]$. 
\end{lemma}
\begin{proof}
	The hypothesis that for all $\varphi \in \test(\dom)$ it holds $\sh{\supp \varphi} \subset \subset \Omega$ ensures the equality
	$$
	\langle f(x), \varphi(x) \rangle = \langle f(x+\varepsilon), \varphi(x+\varepsilon)\rangle
	$$
	from which we deduce the equivalence $f(x) \in \B$ if and only if $f(x+\varepsilon)\in \B$.
	
	We will now prove that, $f(x) \in \B$, then $[f(x)] = [f(x+\varepsilon)]$.
	By equation \ref{party}, we have
	\begin{equation}\label{equation mai citata}
	\langle f(x+\varepsilon)-f(x) , \varphi \rangle = - \langle f(x+\varepsilon), \varepsilon \D \varphi \rangle
	\end{equation}
	for all $\varphi \in \test(\dom)$.
	Notice that $\varphi \in \test(\dom)$ implies that $\varepsilon \D \varphi \in \test(\dom)$ and $\varepsilon \D \varphi(x) \sim 0$ for all $x \in \dom$.
	Hence, by the hypothesis $f \in \B$ and by Lemma \ref{character}, we deduce that $\langle f(x+\varepsilon), \varepsilon \D \varphi \rangle \sim 0$ for all $\varphi \in \test(\dom)$.
	By equation \ref{equation mai citata}, this is sufficient to deduce the equality $[f(x)] = [f(x+\varepsilon)]$.
\end{proof}

As a consequence of the above Lemma, we can characterize a nonstandard counterpart of the shift operator.

\begin{corollary}\label{corollario shift}
	Let $f \in \B$.
	For all $n$ such that $n\varepsilon$ is finite, $[f(x\pm n\varepsilon)] = [f](x\pm\sh{(n\varepsilon)})$. 
\end{corollary}

Thanks to the above results, we can now prove that the finite difference operators generalize the distributional derivative.

\begin{theorem}\label{teorema equivalenza derivate2}
	%Let $I$ be a near-interval.
	The diagrams
	$$
	\begin{array}{ccc}
	\begin{array}{ccc}
	\B & \stackrel{\D^+}{\longrightarrow} & \B \\
	\Phi \circ \pi \downarrow & & \downarrow \Phi \circ \pi\\
	\mbox{}\tests'({\Omega})  & \stackrel{D}{\longrightarrow} & \tests'({\Omega})  \\
	\end{array}
	&
	\text{and}
	&
	\begin{array}{ccc}
	\B & \stackrel{\D^-}{\longrightarrow} & \B \\
	\Phi \circ \pi \downarrow & & \downarrow \Phi \circ \pi\\
	\mbox{}\tests'({\Omega})  & \stackrel{D}{\longrightarrow} & \tests'({\Omega})  \\
	\end{array}
	\end{array}
	$$
	commute.
\end{theorem}
\begin{proof}
	We will prove that the first diagram commutes, as the proof for the second is similar.
	
	Let $f \in \B$: we have the following equality chain
	$$
	\ldual D[f],\varphi \rdual =
	-\ldual [f], D\varphi \rdual =
	-\sh{\langle f, \ns{D}\varphi \rangle}.
	$$
	By Theorem \ref{teorema equivalenza derivate}, $\ns{D}\varphi \equiv_S \D^{\pm} \ns{\varphi}$
	and, by Corollary \ref{cor continuity},
	$$
	\langle f, \ns{D}\varphi\rangle
	\sim
	\langle f, \D^{\pm} \ns{\varphi} \rangle.
	$$
	By the discrete summation by parts formula \ref{party} and by Lemma \ref{shift2} we have
	$$
	\langle f, \D^{\pm} \ns{\varphi} \rangle
	\sim
	- \langle \D^{\pm} f, \ns{\varphi} \rangle
	$$
	from which we deduce
	$$
	\ldual D[f], \varphi \rdual
	=
	\sh{\langle [\D^{\pm} f], \ns{\varphi} \rangle}
	$$
	for all $\varphi \in \tests(\Omega)$.
\end{proof}

By composing finite difference operators, the theorem is easily extended to other differential operators.
We will discuss an example of such a grid function formulation of a differential operator in Section \ref{section nonli}.
Moreover, in \cite{illposed} it is discussed a grid function formulation of the gradient and of the Laplacian.

\subsection{Discrete product rule for generalized distributions and the Schwartz impossibility theorem}\label{section schwartz}

For the usual distributions, it is a consequence of the impossibility theorem by Schwartz that no extension of the distributional derivative satisfies a product rule.
However, for the grid functions there are some discrete product rules that generalize the product rule for smooth functions.
Indeed, the following identities can be established by a simple calculation.

\begin{proposition}[Discrete product rules]\label{chain}
	Let $f, g \in \grid{\dom}$. Then
	\begin{eqnarray*}
		\D^+(f\cdot g)(x) 	& = & \frac{f(x+\varepsilon)g(x+\varepsilon) - f(x)g(x)}{\varepsilon} \\
		& = & f(x+\varepsilon)\D^+ g(x) + g(x)\D^+ f(x) \\
		& = & f(x)\D^+ g(x) + g(x+\varepsilon)\D^+ f(x)
	\end{eqnarray*}
	and
	\begin{eqnarray*}
		\D^-(f\cdot g)(x) 	& = & \frac{f(x)g(x) - f(x-\varepsilon)g(x-\varepsilon)}{\varepsilon} \\
		& = & f(x)\D^- g(x) + g(x-\varepsilon)\D^- f(x) \\
		& = & f(x-\varepsilon)\D^- g(x) + g(x)\D^- f(x).
	\end{eqnarray*}
\end{proposition}

\begin{example}[Derivative of the sign function and the product rule]
	For an in-depth discussion of this example and of the limitations in the definition of a product rule for the distributional derivative, we refer to \cite{tartar}.
	Consider the following representative of the sign function
	$$
	f(x) = \left\{ \begin{array}{rl}
	-1 & \text{if } x<0\\
	1 & \text{if } x \geq 0.
	\end{array}	
	\right.
	$$
	For this function $f$, $f^2 = 1$ and $f^3 = f$, but the distributional derivative $f_x = 2\delta_{0}$ is different from $(f^3)_x = 3 f^2 f_x = 3 f_x= 6 \delta_0$.
	So, even if $f^2$ is smooth, the product rule does not hold.
	
	If we regard $f$ as a grid function, however, the boundedness of $f$ ensures that $f \in \test'(\Lambda)$, and with a simple calculation we obtain:
	\begin{equation}\label{direct calculation}
	\D f(x) = \left\{ \begin{array}{rl}
	2 \varepsilon^{-1} & \text{if } x=-\varepsilon\\
	0 & \text{otherwise}.
	\end{array}	
	\right.
	\end{equation}
	Notice also that $[\D f] = 2 \delta_{0}$, as we expected from Theorem \ref{teorema equivalenza derivate2}.
	%Moreover, since $u = u^3$, we expect that $\D u^3 = \D u$.
	Applying one of the chain rule formulas of Lemma \ref{chain} and taking into account that $f^2(x) = 1$ for all $x \in \Lambda$, we obtain
	\begin{eqnarray*}
		\D f^3(x)
		&=& f(x) \D f^2(x) + f^2(x+\varepsilon) \D f(x)\\
		&=& f(x) (f(x) \D f(x) + f(x+\varepsilon) \D f(x)) + \D f(x)\\
		&=& \D f(x)(2+ f(x) f(x+\varepsilon))
	\end{eqnarray*}
	so that
	$$
	\D f^3(x) = \left\{ \begin{array}{ll}
	\D f(-\varepsilon) = 2\varepsilon^{-1} & \text{if } x=-\varepsilon\\
	0 & \text{otherwise},
	\end{array}	
	\right.
	$$
	in agreement with \ref{direct calculation}.
	%In particular, the discrete chain rule still holds for the derivative of the product of arbitrary grid functions.
\end{example}

We can summarize the results obtained so far as follows: the space of grid functions
\begin{itemize}
	\item is a vector space over $\hR$ that extends the space of distributions in the sense of Theorem \ref{bello};
	\item has a well-defined pointwise multiplication that extends the one defined for $S^0$ functions;
	\item has a derivative $\D$ that generalizes the distributional derivative and for which the discrete version of the chain rule established in Proposition \ref{chain} holds.
\end{itemize}
These properties are the nonstandard, discrete counterparts to the ones itemized in the impossibility theorem by Schwartz \cite{schwartz}.
As a consequence, the space of grid functions can be seen as a non-trivial generalization of the space of distributions, as we claimed at the beginning of this section.

We will complete our study of the relations between the space of grid functions and the space of distributions by showing that the space of distributions can be embedded, albeit in a non-canonical way, in the space of grid functions.
Notice that we cannot ask to this embedding to be fully coherent with derivatives: in fact, there is already an infinitesimal discrepancy between the usual derivative and the discrete derivative in the set of polynomials: the derivative of $x^2$ is $2x$, but $\D x^2 = 2x+\varepsilon$.
%All the differences between the discrete derivative $\D$ and the usual derivative are already present at the level of $C^1$ functions.
However, as shown in Theorem \ref{teorema equivalenza derivate}, for all $f \in C^n$, $D^n f = [\D^n (\ns{f}_{|\Lambda})]$.
In fact, the canonical linear embedding $l : C^0(\R) \hookrightarrow S^0(\Lambda)$ given by $l(f) = \ns{f}_{|\Lambda}$ does not preserve derivatives, but it has the weaker property
\begin{equation}\label{weak agreement}
l(f') \equiv \D (l(f)).
\end{equation}
This will be the weaker coherence request that we will impose on the embedding from the space of distributions to the space of grid functions.

\begin{theorem}\label{schwartz}
	Let $\{\psi_n\}_{n \in \N}$ be a partition of unity, and let $H$ be a Hamel basis for $\tests'(\R)$.
	There is a linear embedding $l : \tests'(\R) \rightarrow \test'(\Lambda)$, that depends on $\{\psi_n\}_{n \in \N}$ and $H$, that satisfies the following properties:
	\begin{enumerate}
		\item $\Phi \circ l = id$;
		\item the product over $\test'(\Lambda) \times \test'(\Lambda)$ generalizes the pointwise product over $C^0(\R) \times C^0(\R)$;
		\item the derivative $\D$ over $\test'(\Lambda)$ extends the distributional derivative in the sense of equation \ref{weak agreement};
		\item the chain rule for products holds in the form established in Lemma \ref{chain}.
	\end{enumerate}
\end{theorem}
\begin{proof}
	We will define $l$ over $H$ and extend it to all of $\tests'(\R)$ by linearity.
	Let $T \in H$.
	From the representation theorem of distributions (see for instance \cite{strichartz}), we obtain
	\begin{equation}\label{struttura distri standard}
	T = \sum_{n \in \N} T\psi_n = \sum_{n \in \N} D^{a_n} f_n
	\end{equation}
	with $f_n \in C^0(\R)$ and $\supp (D^{a_n} f_n) \subseteq \supp \psi_n$ for all $n \in \N$.
	Moreover, the sum is locally finite and for all $\varphi \in \tests(\Omega)$ there exists a finite set $I_{\varphi} \subset \N$ such that
	\begin{equation}\label{locally finite}
	\ldual T, \varphi \rdual
	=
	\ldual \sum_{i \in I_\varphi}D^{a_i} f_i, \varphi \rdual.
	\end{equation}
	
	Let $\{\phi_n\}_{n\in\ns{\N}}$ be the nonstandard extension of the sequence $\{\psi_n\}_{n \in \N}$, and let $\{b_n\}_{n \in \ns{\N}}$ be the nonstandard extension of the sequence $\{a_n\}_{n \in \N}$.
	By transfer, from the representation \ref{struttura distri standard} we obtain
	\begin{equation}\label{struttura distri}
	\ns{T} = \sum_{n \in \ns{\N}} \ns{T}\phi_n = \sum_{n \in \ns{\N}} \ns{D}^{b_n} g_n
	\end{equation}
	with $g_i \in \ns{C^0(\R)}$ and $\supp (D^{b_n} g_n) \subseteq \supp \psi_n$ for all $n \in \ns{\N}$.
	We may also assume that the representation \ref{struttura distri} has the following properties:
	\begin{enumerate}
		\item $	b_n = \min \left\{
		m \in \ns{\N} : \ns{T}\phi_n = \ns{D}^{m} f \text{ with } f \in \ns{C^0(\R)}
		\right\}$
		for all $n \in \ns{\N}$
		\item if $\ns{T}\phi_n = \ns{D}^{b_n} g = \ns{D}^{b_n} h$ with $g, h \in \ns{C^0(\R)}$, then $g-h$ is a polynomial of a degree not greater than $b_n-1$;
		\item if $n$ is finite and $\ns{T}\phi_n = \ns{D}^{b_n}g_n$, then $g_n = \ns{f}_n$ and $b_n = a_n$, where $f_n$ and $a_n$ satisfy $T\psi_n = D^{a_n}f_n$.
	\end{enumerate}
	For $T\in H$, we define
	$$
	l(T) = \sum_{n \in \ns{\N}:\ b_n \leq N} \D^{b_n} ({g_n}_{|\Lambda}),
	$$
	and we extend $l$ to $\tests'(\R)$ by linearity.
	%Nb se $T$ ha supporto cpto questa è semplicemente la somma finita solita.
	Notice that $l$ does not depend on the choice of the functions $\{g_n\}_{n \in \ns{\N}}$.
	In fact, suppose that $\ns{T}\phi_n = \ns{D}^{b_n} g = \ns{D}^{b_n} h$ with $g, h \in \ns{C^0(\R)}$.
	By property (2) of the representation \ref{struttura distri}, $g-h$ is a polynomial of a degree not greater than $b_n-1$.
	Recall that, if $p \in \grid{\Lambda}$ is a polynomial of degree at most $b_n-1$, then $\D^{b_n} p = 0$.
	As a consequence, $\D^{b_n} ({g}_{|\Lambda}) = \D^{b_n}({h}_{|\Lambda})$, as we wanted.
	
	We will now show that, for all $T \in H$, $\ldual \Phi([l(T)]), \varphi \rdual = \ldual T, \varphi\rdual$ for all $\varphi \in \tests'(\R)$.
	This equality and linearity of $l$ entail that $\Phi \circ l = id$.
	Let $\varphi \in \tests(\R)$, and let $I_\varphi \subset \N$ a finite set such that equality \ref{locally finite} holds.
	We claim that whenever $i \not \in I_\varphi$, then	$\langle \D^{b_i} ({g_i}_{|\Lambda}), \ns{\varphi} \rangle = 0$.
	In fact, if $i \not \in I_\varphi$ is finite, then by formula \ref{locally finite} and by property (3) of the representation \ref{struttura distri} we have
	$$
	\sh{\langle \D^{b_i} ({g_i}_{|\Lambda}), \varphi} \rangle = \sh{\langle \D^{a_i} ({\ns{f}_i}_{|\Lambda}), \varphi} \rangle =  \ldual D^{a_i} f_i, \varphi \rdual = 0.
	$$
	
	We want to show that $\langle \D^{b_i} ({g_i}_{|\Lambda}), \ns{\varphi} \rangle = 0$ also when $i$ is infinite.
	Notice that if $x \in \fin$, then for sufficiently large $n \in \N$ it holds $x \not \in \supp \phi_n$: otherwise, we would also have $\sh{x} \in \supp \psi_n$ for arbitrarily large $n$, against the fact that for all $x \in \fin$, $\sh{x}\in\supp \phi_n$ only for finitely many $n$.
	As a consequence, $\supp \phi_i \cap \fin = \emptyset$, and by the inclusion $\supp (D^{b_i} g_i) \subseteq \supp \phi_i$, then also $\supp (D^{b_i} g_i) \cap \fin = \emptyset$.
	Taking into account property (2) of the representation \ref{struttura distri}, we deduce that the restriction of $g_i$ to $\hR \setminus \supp (D^{b_i} g_i)$ is a polynomial $p$ of degree at most $b_n -1$.
	%From this result, from equation \ref{derivata discreta ordine M} and by recalling that $b_n \leq N$, we deduce that also $\supp (\D^{b_i} {(g_i)}_{|\Lambda}) \cap \fin = \emptyset$.
	We have already observed that $\D^{b_i} p = 0$ and, as a consequence, $\sh{\langle \D^{b_i} ({g_i}_{|\Lambda}), \varphi} \rangle = 0$.
	
	We then have the following equality:
	$$
	\langle l(T), \ns{\varphi} \rangle = \langle \sum_{i \in I_\varphi}\D^{a_i} ({\ns{f}_i}_{|\Lambda}), \ns{\varphi} \rangle.
	$$
	By Theorem \ref{teorema equivalenza derivate2}, we obtain
	$$
	%\begin{array}{rcl}
	\langle l(T), \ns{\varphi} \rangle =
	%&=&\displaystyle
	\langle \sum_{i \in I_\varphi}\D^{a_i} ({\ns{f}_i}_{|\Lambda}), \ns{\varphi} \rangle =
	%&\sim&\displaystyle
	\ldual \sum_{i \in I_\varphi}D^{a_i} f_i, \varphi \rdual =
	%&=&
	\ldual T, \varphi\rdual,
	%\end{array}
	$$
	that is sufficient to conclude that $\Phi([l(T)]) = T$.
	
	Assertion (2) is a consequence of Lemma \ref{lemma piccolo}, assertion (3) is a consequence of Theorem \ref{teorema equivalenza derivate2}, and assertion (4) is a consequence of Proposition \ref{chain}.
\end{proof}

\section{Grid functions as $\ns{L}^p$ functions and as parametrized measures}\label{sez young}

The main goal of this section is to show that there is an external $\fin$-submodule of the space of grid functions whose elements correspond to Young measures, and that this correspondence is coherent with the projection $\Phi$ defined in Theorem \ref{bello}.
Moreover, we will show how this correspondence can be generalized to arbitrary grid functions.
Before we prove these results, we find it useful to discuss some properties of grid functions when they are interpreted as $\ns{L^p}$ functions.
These properties are interesting on their own, and will also be used also in Section \ref{solutions}, when we will discuss the grid function formulation of partial differential equations.

Recall that for all $1 \leq p \leq \infty$, a function $g \in L^p(\Omega)$ induces a distribution $T_g \in \tests'(\Omega)$ defined by
$$
\ldual T_g, \varphi \rdual = \int_{\Omega} g \varphi dx 
$$
for all $\varphi \in \tests(\Omega)$.
As a consequence, by identifying $g$ with $T_g$ we have the inclusions $L^p(\Omega) \subset \tests'(\Omega)$ for all $1 \leq p \leq \infty$.
Since $\Phi$ is surjective, we expect that for all $g \in L^p(\Omega)$ there exists $f \in \B$ satisfying $[f] = T_g$.
In this case, we will often write $[f] = g$ and $[f] \in L^p(\Omega)$.
If $[f] \in L^p(\Omega)$, thanks to the Riesz representation theorem, we can think of $[f]$ either as a functional acting on $L^{p'}(\Omega)$, or as a member of an equivalence class of $L^p(\Omega)$ functions.
To our purposes, we find it more convenient to treat $[f]$ as a function.
With this interpretation, if $g = [f]$ and $g\in L^p(\Omega)$, then it holds the equality $g(x) = [f](x)$ for almost every $x \in \Omega$. 

\subsection{Grid functions as $\ns{L}^p$ functions}\label{sub lp}

%For a matter of commodity, we will assume that $\Omega \subset \R^n$ bounded.
We can identify every grid function with a piecewise constant function defined on all of $\hR^k$.
Among many different identifications, we choose the following: if $f \in \grid{\dom}$, then $\lp{f}$ is defined by
$$
\lp{f}(x) = \left\{
\begin{array}{ll}
f((n_1,n_2,\ldots,n_k)\varepsilon) & \text{if } n_i \varepsilon \leq x_i < (n_i+1)\varepsilon \text{ for all } 1 \leq i \leq k\\
0 & \text{if } |x_i| > N \text{ for some } 	1 \leq i \leq k,
\end{array}
\right.
$$
with the agreement that $f((n_1,n_2,\ldots,n_k)\varepsilon) = 0$ if $(n_1,n_2,\ldots,n_k)\varepsilon \not \in \dom$.

If $f$ is a grid function, the function $\lp{f}$ is an internal $\ns{}$simple function and, as such, it belongs to $\ns{L^p}(\R^k)$ for all $1 \leq p \leq \infty$.
The integral of $\lp{f}$ is related with the grid integral of $f$ by the following formula:
$$
\ns{\int}_{\hR^k} \lp{f} dx = \int_{\dom} f(x) d\Lambda^k = \varepsilon^k \sum_{x \in \dom} f(x).
$$
As a consequence, the $\ns{L^p}$ norm of $\lp{f}$ can be expressed by
$$
\norm{\lp{f} }_p^p = \varepsilon^k \sum_{x \in \dom} |f(x)|^p \text{ if } 1 \leq p < \infty, \text{ and } \norm{\lp{f}}_\infty = \max_{x \in \dom} |f(x)|.
$$

Notice that if $f \in \grid{\dom}$, then $\sh{\supp \lp{f}} \subseteq \sh{\supp \lp{\chi_{\dom}}} = \overline{\Omega}$.
If we define $\lp{\Omega} = \supp \lp{\chi_{\dom}}$, then from the above inclusion we can write $\lp{f} \in \ns{L}^p(\lp{\Omega})$ for all $1 \leq p \leq \infty$.
By identifying $f \in \grid{\dom}$ with $\lp{f}$, for all $1 \leq p \leq \infty$ the space %$\grid{\dom}$
of grid functions is identified with a subspace of $\ns{L^p(\lp{\Omega})}$ which is closed with respect to the $\ns{L}^p$ norm.
Since $\lp{\Omega}$ is $\ns{}$bounded in $\hR^k$, for $1 \leq p \leq \infty$ we have the usual relations between the $\ns{L^p}$ norms of $f \in \grid{\dom}$: if $1\leq p < q \leq +\infty$ and if $r$ satisfies the equality $1/p = 1/q + 1/r$, then
$\norm{\lp{f}}_p \leq \ns{\leb(\lp{\Omega})}^r\norm{\lp{f}}_q$.

From now on, when there is no risk of confusion, we will often abuse the notation and write $f$ instead of $\lp{f}$.

We begin our study of grid functions as $\ns{L^p}$ functions by showing that if a grid function $f$ has finite $\ns{L}^p$ norm for some $1\leq p\leq \infty$, then $f\in\B$ and, as a consequence, $[f]$ is a well-defined distribution.

\begin{lemma}\label{lemma sufficiente}
	If $\norm{f}_p \in \fin$ for some $1 \leq p \leq \infty$, then $f \in \test'(\dom)$.
\end{lemma}
\begin{proof}
	Notice that $\test(\dom)\subset \ns{L^p}(\lp{\Omega})$ for all $1 \leq p \leq \infty$ and, for any $\varphi \in \test(\dom)$, $\norm{\varphi}_p \in \fin$ for all $1 \leq p \leq \infty$.
	By the discrete H\"older's inequality
	$$
	|\langle f, \varphi \rangle| \leq \norm{f\varphi}_1 \leq \norm{f}_p \norm{\varphi}_{p'}
	$$
	%for all $\varphi \in \test(\dom)$,
	so that if $\norm{f}_p \in \fin$, then $\langle f, \varphi \rangle \in \fin$ for all $\varphi \in \test(\dom)$, as desired.
\end{proof}

From the previous Lemma we deduce that, if the $L^p$ norm of the difference of two grid functions $f$ and $g$ is infinitesimal, then $f \equiv g$.

\begin{corollary}\label{corollario lp}
	Let $f, g \in \grid{\dom}$.
	If $\norm{f-g}_p \sim 0$ for some $1 \leq p \leq \infty$, then $f \equiv g$.
\end{corollary}
\begin{proof}
	If $\norm{f-g}_p \sim 0$, then by Lemma \ref{lemma sufficiente}
	$$
	\langle f-g, \varphi \rangle \leq \norm{f-g}_p \norm{\varphi}_{p'} \sim 0
	$$
	for all $\varphi \in \test(\dom)$.
	As a consequence, $f \equiv g$.
\end{proof}

Notice that the other implication does not hold, in general.
As an example, consider the grid function
$f(n\varepsilon) = (-1)^n$.
Since $\langle f, \varphi \rangle \sim 0$ for all $\varphi \in \test(\dom)$, we deduce that $[f] = 0$, but $\norm{f}_{p} = 1$ for all $1 \leq p \leq \infty$.
Notice also that $\norm{f}_p$ is finite, but $\norm{\lp{f}-\ns{g}}_p \not \sim 0$ for all $g \in L^p(\Omega)$ and for all $1 \leq p \leq \infty$.

In the next section, we will show that the hypothesis $\norm{f}_\infty \in \fin$ is sufficient to ensure that $[f]\in L^\infty(\Omega)$.
If $1 \leq p < \infty$, however, the hypothesis $\norm{f}_p \in \fin$ is not sufficient to imply that $[f] \in L^p(\Omega)$.
An example is given by $N\chi_0 \in\grid{\Lambda}$, a representative of the Dirac distribution centred at $0$.
It can be calculated that $\norm{N\chi_0}_1 = \varepsilon N = 1,$
but $[N\chi_0] = \delta_0 \not\in L^p(\R)$ for any $p$.
In general, whenever $[f]\in L^p(\Omega)$, it holds the inequality $\norm{f}_p \geq \norm{[f]}_p$.

\begin{proposition}\label{norm inequality}
	For all $f \in \grid{\dom}$ and for all $1 \leq p \leq \infty$, if $[f] \in L^p(\Omega)$, then
	\begin{enumerate}
		\item if $[|f|] \in L^p(\Omega)$, then $[|f|]\geq|[f]|$ a.e.\ in $\Omega$;
		\item $\sh{\norm{f}_p} \geq \norm{[f]}_p$.
	\end{enumerate}
\end{proposition}
\begin{proof}
	Define $f^+(x) = \max\{f(x),0\}$ and $f^-(x) = \min\{f(x),0\}$, so that $f = f^+ + f^-$ and $|f|^p = |f^+|^p + |f^-|^p$ for all $1 \leq p < \infty$.
	If $[|f|]\in L^p(\Omega)$, then $[f^+]$ and $[f^-] \in L^p(\Omega)$ and, by linearity of $\Phi$,
	$$
	[|f|](x) = [f^+](x)-[f^-](x) \geq [f^+](x)+[f^-](x) = [f](x)
	$$
	for a.e.\ $x \in\Omega$.
	
	Let $f\in\grid{\dom}$ and suppose that $[f]\in L^p(\Omega)$ with $p < \infty$.
	If either $|f^+| \not \in \B$, $|f^+|^p \not \in \B$, $|f^-| \not \in \B$ or $|f^-| \not \in \B$ %then $|f^+|^p \not \in \B$ or $|f^-|^p \not \in \B$.
	%Suppose without loss of generality that $|f^+|^p \not \in \B$:
	then by Lemma \ref{lemma sufficiente} we would have $|f|^p \not \in \B$ and, as a consequence,
	$$\norm{f}_p^p = \norm{|f|^p}_1 \not \in \fin,$$
	so that inequality (2) would hold.
	Suppose then that $|f^+| \in \B$, $|f^+|^p \in \B$, $|f^-| \in \B$ and $|f^-|^p \in \B$.
	As a consequence, both $|f| \in \B$ and $|f|^p \in \B$.
	If $[|f|]\in L^p(\Omega)$, then (2) is a consequence of (1).
	The only case left is $[|f|]\not\in L^p(\Omega)$.
	
	For a matter of commodity, let $g = [f]$, and let $g^+(x) = \max\{g(x),0\}$ and $g^-(x) = \min\{g(x),0\}$.
	Since
	$
	[f^+]+[f^-] = [f] = g^++g^- \text{ in } \tests'(\Omega),
	$
	we deduce that
	$
	[f^+]-g^+ = -([f^-]-g^-).
	$
	The hypothesis $[f^+]\not \in L^p(\Omega)$ entails that also $[f^+]-g^+ \not \in L^p(\Omega)$.
	Let $K = \supp([f^+]-g^+)$: then for all $\varphi \in \tests(\Omega)$ with $\supp\varphi \subset K$ and with $\varphi(x) \geq 0$ for all $x \in \Omega$,
	$$
	0 \leq \ldual [f^+]-g^+, \varphi \rdual = \sh{\langle f^+, \ns{\varphi}\rangle} - \int_{\Omega} g^+ \varphi dx.
	$$
	Similarly,
	$$
	0 \leq -\ldual [f^-]-g^-, \varphi \rdual = \sh{\langle |f^-|, \ns{\varphi}\rangle} - \int_{\Omega} |g^-| \varphi dx.
	$$
	From the arbitrariness of $\varphi$, we deduce
	$ \norm{f\chi_{K_\Lambda}}_p \geq \norm{g\chi_K}_p$.
	Since $K = \supp([f^+]-g^+)$, we also have
	$$
	\norm{([f^+]-g^+)\chi_{\Omega\setminus K}}_p = \norm{([f^-]-g^-)\chi_{\Omega\setminus K}}_p = \norm{0}_p = 0,
	$$
	from which we conclude that (2) indeed holds.
	
	Suppose now that $[f]\in L^\infty(\Omega)$.
	If $\norm{f}_\infty \not\in\fin$, then inequality (2) holds.
	If $\norm{f}_\infty \in \fin$, let $c_f \in \grid{\dom}$ satisfy $c_f(x) = \norm{f}_\infty$ for all $x \in \dom$.
	Then $[c_f](x) = \sh{\norm{f}_\infty}$ for all $x \in \Omega$, so that $[c_f] \in L^\infty(\Omega)$.
	Since $c_f(x) \geq \max\{f^+(x), |f^-(x)|\}$ for all $x \in \dom$, then also $[c_f](x) \geq [f](x)$ for all $x \in \dom$.
	This is sufficient to conclude that inequality (2) holds.
\end{proof}

If $[f] \in L^p(\dom)$ and $\sh{\norm{f}_{p}} > \norm{[f]}_{p}$, then $f$ features some oscillations that are compensated by the linearity of $\Phi$.
In this case, we can interpret $f$ as the representative of a weak or (weak-$\star$ when $p = \infty$) limit of a sequence of functions whose $L^p$ norm is uniformly bounded by $\sh{\norm{f}_{p}}$.
In the next section, we will %discuss the case $p = \infty$ and
see how the behaviour of this weak-$\star$ limit can be described by a parametrized measure associated to $f$.

If $\norm{f}_p \not\in \fin$ but nevertheless $[f] \in L^p(\Omega)$, then $f$ also features concentrations that are compensated by the linearity of $\Phi$.
An example is given by the function $f = \D \chi_0 = N\chi_{-\varepsilon}-N\chi_0$.
The $\ns{L}^p$ norm of $f$ is $\norm{f}_p = 2N^{p-1/p}$ for $p \not = \infty$ and $N$ for $p = \infty$; however, from Theorem \ref{teorema equivalenza derivate2}, we deduce that $[f]=D[\chi_0]=0$.
In the next section, we will discuss how these concentrations affect the parametrized measure associated to $f$.
%We will see in the next section that the oscillations of $f$ can be understood by using Young measures.

We will now address the coherence between the nonstandard extension of a $L^2$ function and its projection in the space of grid functions.
These technical results will be used in Section \ref{solutions}.% in order to formulate a grid function approach to nonlinear partial differential equations.

\begin{definition}
	Let $\lpi : \ns{L^2(\lp{\Omega})} \rightarrow \grid{\dom}$ be the $\ns{L^2}$ projection over the closed subspace $\grid{\dom}$.
	Recall that $\lpi(f)$ is the unique grid function satisfying
	$$
	\langle \lpi(f), g \rangle = \ns{\int}_{\lp{\Omega}} f(x) \lp{g}(x) dx
	$$
	for all $g \in \grid{\dom}$.
\end{definition}

\begin{lemma}\label{proiez s-c}
	For all $f \in C^0(\Omega)$,
	$\lpi(\ns{f}) \in S^0(\dom)$ and $\ns{f}(x) \sim \lpi(\ns{f})(x)$ for all $x \in \dom$.
\end{lemma}
\begin{proof}
	Let $f \in C^0(\Omega)$.
	Since for all $g \in\grid{\dom}$ we have the equality
	$$
	\langle \lpi(\ns{f}), g\rangle
	=
	\ns{\int}_{\lp{\Omega}} \ns{f}(x) \lp{g}(x) dx,
	$$
	by choosing $g = \varepsilon^{-k} \chi_{y}$ , we obtain
	$$
	\lpi(\ns{f})(y)
	=
	\langle \lpi(\ns{f}), \lp{\varepsilon^{-k} \chi_{y}}\rangle
	=
	\varepsilon^{-k}\ns{\int}_{[y, y+\varepsilon]^k} \ns{f}(x) dx
	$$
	for all $y \in\dom$.
	Since
	$$
	\min_{x \in [y, y+\varepsilon]^k} \{\ns{f}(x)\}
	\leq
	\varepsilon^{-k}\ns{\int}_{[y, y+\varepsilon]^k} \ns{f}(x) dx
	\leq
	\max_{x \in [y, y+\varepsilon]^k} \{\ns{f}(x)\},$$
	by S-continuty of $\ns{f}$, we deduce the thesis.
\end{proof}

\begin{lemma}\label{questo corollario}
	For all $f \in L^2(\Omega)$, $[\lpi(\ns{f})] = f$.
\end{lemma}
\begin{proof}
	For all $\varphi \in \tests'(\Omega)$ we have
	$$
	\langle \lpi(\ns{f}), \ns{\varphi}_{|\Lambda} \rangle = \ns{}\int_{\lp{\Omega}} \ns{f} \lp{\ns{\varphi}_{|\Lambda}} dx
	$$
	and, by S-continuity of $\ns{\varphi}$,%, since $\ns{\phi}_{|\Lambda} \in \test(\dom) \subset \grid{\dom}$,
	$$
	%\int_{\ns{\Omega}} \lpi(\ns{f}) \ns{\phi}_{|\Lambda} dx
	%=
	\ns{\int}_{\lp{\Omega}} \ns{f} \lp{\ns{\varphi}_{|\Lambda}} dx
	\sim
	\ns{\int}_{\ns{\Omega}} \ns{f} \ns{\varphi} dx
	=
	\int_{\Omega} f\varphi dx.
	$$
	This implies $[\lpi(\ns{f})] = f$.
\end{proof}

The above Lemma can be sharpened under the hypothesis that $\Omega$ has finite measure.

\begin{lemma}\label{norma r}
	Let $\mu_L(\Omega)<+\infty$.
	For all $f \in L^2(\Omega)$,
	$\norm{\ns{f}-\lpi(\ns{f})}_2 \sim 0$.
\end{lemma}
\begin{proof}
	Let $f \in L^2(\Omega)$, and let $r = \ns{f}-\lpi(\ns{f})$.
	By the properties of the $\ns{L^2}$ projection, we have
	\begin{equation}\label{lpproj}
	\norm{\ns{f}}_2 = \norm{P(\ns{f})}_2 + \norm{r}_2.
	\end{equation}
	\begin{comment}
	By Lusin's Theorem, for all $\eta \in \R$, $\eta > 0$, there is a compact set $K_\eta$ such that $f_{|K_\eta}$ is continuous and $|\Omega \setminus K_\eta| < \eta$.
	Hence, $\ns{f}$ is S-continuous over $\ns{K_\eta}$.
	By Lemma \ref{proiez s-c}, $\norm{\ns{f} \chi_{\ns{K_\eta}}}_2 \sim \norm{\lpi(\ns{f}) \chi_{\ns{K_\eta}}}_2$ and, %since
	%$$
	%\norm{\ns{f} \chi_{\ns{K_\eta}}}_2 = \norm{\lpi(\ns{f}) \chi_{\ns{K_\eta}}}_2 + \norm{r \chi_{\ns{K_\eta}}}_2,
	%$$
	by equality \ref{lpproj}, $\norm{r \chi_{\ns{K_\eta}}}_2 \sim 0$ for all $\eta > 0$.
	\end{comment}
	By the nonstandard Lusin's Theorem, there exists a $\ns{}$compact set $K\subseteq \ns{\Omega}$ that satisfies $\ns{\leb}(\ns{\Omega} \setminus K) \sim 0$ and $\norm{r \chi_{K}}_2 \sim 0$.
	Since $\ns{\leb}(\ns{\Omega} \setminus K) \sim 0$ and since $f \in L^2(\Omega)$, we have also $\norm{\ns{f} \chi_{K}}_2 \sim \norm{\ns{f}}_2$ and, as a consequence,
	$$
	\norm{\ns{f} }_2 \sim \norm{\ns{f} \chi_{K}}_2 = \norm{\lpi(\ns{f}) \chi_{K}}_2 + \norm{r\chi_K} \sim \norm{\lpi(\ns{f}) \chi_{K}}_2.
	$$
	From the inequality chain
	$$\norm{\ns{f} }_2\sim\norm{\lpi(\ns{f}) \chi_{K}}_2 \leq \norm{\lpi(\ns{f})}_2 \leq \norm{\ns{f}}_2$$
	we deduce that $\norm{\ns{f}}_2 \sim \norm{P(\ns{f})}_2$ that, by equality \ref{lpproj}, implies $\norm{\ns{f}-\lpi(\ns{f})}_2 \sim 0$, as we wanted.
\end{proof}

The previous Lemma suggests a definition of nearstandardness that will be useful in the sequel of the paper.

\begin{definition}
	Let $\mu_L(\Omega)<+\infty$.
	We will say that $f \in \grid{\dom}$ is nearstandard in $L^2(\Omega)$ iff there exists $g \in L^2(\Omega)$ such that $\norm{f-P(\ns{g})}_2 \sim 0$.
\end{definition}

Notice that, under the hypothesis that $\mu_L(\Omega)$ is finite, Corollary \ref{corollario lp} and Lemma \ref{norma r} entail that $f$ is nearstandard in $L^2(\Omega)$ if and only if $[f] \in L^2(\Omega)$ and $\norm{f-P(\ns{[f]})}_2 \sim 0$.

\subsection{An extension of the Robinson-Bernstein embedding}
We conclude the study of the properties of grid functions as $\ns{L^p}$ functions by discussing %showing that , thanks to the embedding $l$ defined in Theorem \ref{schwartz} and to the identification between $f$ and $\lp{f}$,
the generalization of an embedding due to Robinson and Bernstein
$$
L^2(\Omega) \subset V \subset \ns{L^2}(\Omega),
$$
where $V$ is a vector space of a hyperfinite dimension (for the details, we refer to \cite{invariant, da}).
%Recall that, since vector spaces over the same field and with the same dimension are isomorphic, $V$ is isomorphic to the space of grid functions obtained by choosing $N$ equal to the dimension of $V$.
In our case, by considering the embedding $l$ of the space of distributions to the space of grid functions defined in Theorem \ref{schwartz} and by modifying the extension of
%of $\grid{\dom}$ as a subspace of $\ns{L^p}(\lp{\Omega})$
$f$ to $\lp{f}$, we will obtain the inclusions
$$
L^p(\Omega) \subset \tests'(\Omega) \subset \grid{\dom} \subset \ns{L^p}(\Omega)
$$
for all $1\leq p \leq \infty$.

\begin{proposition}\label{proposition embedding}
	Let $l$ be defined as in the proof of Theorem \ref{schwartz}.
	There is an embedding
	%$l: L^p(\Omega) \rightarrow \grid{\dom}$ and
	$l': \grid{\dom} \rightarrow \bigcap_{1\leq p \leq \infty} \ns{L^p}(\Omega)$
	such that
	\begin{equation}\label{embedding equality}
	\ns{\int}_{\hR^k} (l'\circ l)(f) \ns{\varphi} dx \sim \int_{\R^k} f \varphi dx
	\end{equation}
	for all $1 \leq p \leq \infty$, for all $f \in L^p(\Omega)$ and for all $\varphi \in \tests(\Omega)$.
	As a consequence, if we identify $\tests'(\Omega)$ with $l(\tests'(\Omega)) \subseteq \grid{\dom}$ and $\grid{\dom}$ with $l'(\grid{\dom}) \subseteq \ns{L^p}(\Omega)$, we have the inclusions
	$$
	L^p(\Omega) \subset \tests'(\Omega) \subset \grid{\dom} \subset \ns{L^p}(\Omega)
	$$
	for all $1\leq p \leq \infty$.
\end{proposition}
\begin{proof}
	%Let $l$ be the embedding given by Theorem \ref{schwartz}, and 
	Define $l'$ by $l'(f) = \lp{f}\chi_{\ns{\Omega}}$ for all $f \in \grid{\dom}$.
	Since $l'(f)$ is an internal $\ns{}$simple function, it belongs to $\ns{L^p}(\Omega)$ for all $1 \leq p \leq \infty$.
	We will now prove that, for this choice of $l'$, equality \ref{embedding equality} holds.
	
	Notice that for all $f \in \grid{\dom}$, if $l'(f)(x) \not = \lp{f}(x)$, then $x \in \ns{\Omega}\setminus\lp{\Omega}$ or $x \in \lp{\Omega}\setminus\ns{\Omega}$.
	By the definition of $\lp{\Omega}$, this entails $\sh{x} \in \partial\Omega$.
	In particular, if $\varphi \in \tests(\Omega)$, then $\sh{x} \not \in \supp \varphi$.
	As a consequence, for all $f \in \B$ and for all $\varphi \in \tests(\Omega)$, it holds
	$$\ns{\int}_{\hR^k} l'(f)\ \ns{\varphi} dx = \ns{\int}_{\hR^k} \lp{f}\ \ns{\varphi} dx.$$
	By S-continuity of $\ns{\varphi}$, we have also
	$$\ns{\int}_{\hR^k} \lp{f}\ \ns{\varphi} dx \sim \langle f, \ns{\varphi} \rangle.$$
	If we let $f = l(g)$ for some $g \in L^p(\Omega)$, from Theorem \ref{schwartz} we have
	$$
	\langle l(g),\ns{\varphi} \rangle \sim \ldual g, \varphi \rdual = \int_{\R^k} g\varphi dx.
	$$
	By putting together the previous equalities, we conclude that equation \ref{embedding equality} holds.
\end{proof}

We conjecture that for $p = 2$ and under the hypothesis that $\Omega$ has finite Lebesgue measure, we can choose the embedding $l$ in a way that the equality $\norm{(l' \circ l)(f) - \ns{f}}_2 \sim 0$ holds, as in the original embedding by Robinson and Bernstein.

\subsection{Grid functions as parametrized measures}

It is well known that weak limits of $L^p$ functions behave badly with respect to composition with a nonlinear function \cite{balder, evans nonlinear, sychev, webbym}.
Consider for instance a bounded sequence $\{ f_n \}_{n \in \N}$ of $L^\infty(\Omega)$ functions: by the Banach–Alaoglu theorem, there is a subsequence of $\{f_n\}_{n \in \N}$ that has a weak-$\star$ limit $f_\infty \in L^\infty(\Omega)$.
Now let $\Psi \in \bcf(\R)$: the sequence $\{ \Psi(f_n)\}_{n \in \N}$ is still bounded in $L^\infty(\Omega)$, so it has a weak-$\star$ limit $\Psi_\infty$.
However, in general $\Psi_\infty \not = \Psi(f_\infty)$.
To overcome this difficulty, the weak-$\star$ limit of the sequence $\{ f_n \}_{n \in \N}$ can be represented by a Young measure.
In particular, the main theorem of Young measures states that for every bounded sequence $\{ f_n \}_{n \in \N}$ of $L^\infty(\Omega)$ functions there exists a measurable function $\nu : \Omega \rightarrow \prob(\R)$ that satisfies the following property: for all $\Psi \in \bcf(\R)$, the weak-$\star$ limit of $\{ \Psi(f_n) \}_{n \in \N}$ is the function defined by $\overline{\Psi}(x) = \int_{\R} \Psi d\nu_x$, in the sense that the equality
\begin{equation}\label{young definition}
\lim_{n \rightarrow \infty} \int_{\Omega} \Psi(f_n(x)) \varphi(x) dx
=
\int_{\Omega} \left(\int_{\R} \Psi d\nu_x\right) \varphi(x) dx
=
\int_{\Omega} \overline{\Psi}(x) \varphi(x) dx
\end{equation}
holds for all $\varphi \in L^1(\Omega)$.

\begin{example}
	The following example is discussed in detail in \cite{webbym}.
	Consider the Rademacher functions $f_n(x) = f_0(n^2x)$, with $f_0(x) = \chi_{[0,1/2)}(x) - \chi_{[1/2,1)}(x)$ extended periodically over $\R$.
	It can be calculated that the Young measure $\nu$ associated to the sequence $\{f_n\}_{n \in \N}$ is constant and that
	$$
	\nu_x = \frac{1}{2} \delta_{1} + \frac{1}{2} \delta_{-1}
	$$
	for almost every $x \in \Omega$, i.e.\ that for all $\Psi \in \bcf(\R)$ and for all $\varphi \in L^1(\Omega)$,
	$$
	\lim_{n \rightarrow \infty} \int_{\R} \Psi(f_n(x)) \varphi(x) dx
	=
	\left(\frac{1}{2}\Psi(1) + \frac{1}{2}\Psi(-1)\right) \int_{\R} \varphi(x) dx.
	$$
\end{example}

In the setting of grid functions, instead of bounded sequences of $L^\infty$ functions, we have grid functions with finite $\ns{L^\infty}$ norm.
These functions can be used to represent weak-$\star$ limits of $L^\infty$ functions.

\begin{example}
	The function $f(n\varepsilon) = (-1)^n$ can be thought as a grid function representative for the weak-$\star$ limit of the Rademacher functions: in fact, for all $\Psi \in \bcf(\R)$ and for all $\varphi \in C^0_c(\Omega)$,
	$$
	\sh{\langle \ns{\Psi}(f), \ns{\varphi} \rangle} = \left(\frac{1}{2}\Psi(1) + \frac{1}{2}\Psi(-1)\right) \int_{\R} \varphi(x) dx.
	$$
	Since $C^0_c(\Omega)$ is dense in $L^1(\Omega)$, this is sufficient to conclude that the above formula holds for all $\varphi \in L^1(\Omega)$.
\end{example}

In \cite{cutland controls3}, Cutland showed that every grid function that has finite $\ns{L}^\infty$ norm corresponds to a Young measure.

\begin{theorem}\label{young}
	For every $f \in \grid{\dom}$ with $\norm{f}_\infty \in \fin$, there exists a Young measure $\nu^f : \Omega \rightarrow \prob(\R)$ such that, for all $\Psi \in \bcf(\R)$ and for all $\varphi \in C^0_c(\Omega)$,
	\begin{equation}\label{young equivalence equation}
	\sh{\langle \ns{\Psi}(f), \ns{\varphi} \rangle}
	=
	\int_{\Omega} \left( \int_{\R} \Psi d \nu^f_x \right) \varphi(x) dx.
	\end{equation}
\end{theorem}
\begin{proof}
	Since $\norm{f}_\infty \in \fin$, there exists $n \in \R$ such that $|f(x)| < n$.
	We can identify $f$ with a function $\tilde{f} : \lp{\Omega} \rightarrow \ns{\prob(\ns{[-n,n]})}$ defined by
	$\tilde{f}(x) = \delta_{\lp{f}(x)}$.
	Notice that for all $\Psi \in \bcf(\R)$ and for all $\varphi \in C^0_c(\Omega)$ it holds
	\begin{eqnarray}
	\notag
	\langle \ns{\Psi(f)}, \ns{\varphi} \rangle
	&\sim&
	\ns{\int_{\lp{\Omega}}} \ns{\Psi(\lp{f}(x))} \ns{\varphi(x)} dx\\
	&=&\label{quiqui}
	\ns{\int_{\lp{\Omega}}} \left(\ns{\int}_{\ns{[-n,n]}} \ns{\Psi} d \tilde{f}(x) \right) \ns{\varphi}(x)dx.
	\end{eqnarray}
	
	We define an internal measure $\mu$ over $\ns{\Omega} \times \ns{[-n,n]}$ by posing
	$$
	\mu(A\times B) = \ns{\int}_A \tilde{f}_x(B) dx
	$$
	for all Borel $A \subseteq \Omega$ and for all Borel $B \subseteq \ns{[-n,n]}$.
	Let $L_\mu$ be the Loeb measure obtained from $\mu$ (for the properties of the Loeb measure, we refer for instance to \cite{nsa theory apps, lo, nsa working math}).
	We can define a standard measure $\mu_s$ over $\Omega \times [-n,n]$ by posing
	$$
	\mu_s(A \times B) = L_\mu(\{x \in \ns{\Omega} \times \ns{[-n,n]} : \sh{x} \in A \times B\}).
	$$
	Since $\mu_s$ satisfies $\mu_s(A \times [-n,n]) = \leb(A)$ for all Borel $A \subseteq \Omega$, by Rohlin's Disintegration Theorem the measure $\mu_s$ can be disintegrated as
	$$
	\mu_s(A \times B) = \int_A \nu^f_x(B) dx,
	$$
	with $\nu^f : \Omega \rightarrow \prob([-n,n])$.
	By Lemma 2.6 of \cite{cutland controls3}, $\nu^f$ satisfies
	$$
	\sh{\left(\ns{\int_{\ns{\Omega}}} \left(\ns{\int}_{\ns{[-n,n]}} \ns{\Psi} d \tilde{f}(x) \right) \ns{\varphi}(x)dx\right)}
	=
	\int_\Omega \left(\int_{[-n,n]} \Psi d \nu^f_x\right) \varphi(x) d x.
	$$
	for all $\Psi \in \bcf(\R)$ and for all $\varphi \in C^0_c(\Omega)$.
	Thanks to equality \ref{quiqui}, we deduce that $\nu^f$ satisfies \ref{young equivalence equation}.
	We can extend $\nu^f_x$ to all of $\prob{(\R)}$ by defining $\nu^f_x(A) = \nu^f_x(A\cap[-n,n])$ for all Borel sets $A \subseteq \R$ and for all $x \in \Omega$, thus obtaining a Young measure that satisfies equation \ref{young equivalence equation}.
\end{proof}

In \cite{ball, webbym}, it is shown that Young measures describe weak-$\star$ limits of bounded sequences of $L^\infty$ functions.
We will now show that grid functions with finite $L^\infty$ norm can be similarly used to represent weak-$\star$ limits of $L^\infty$ functions in the setting of grid functions.
This is a consequence of a more general property of the correspondence between grid functions and Young measures: if $f \in \grid{\dom}$ satisfies $\norm{f}_\infty \in \fin$ and $\nu^f$ is the Young measure associated to $f$ in the sense of Theorem \ref{young}, then $[f]$ corresponds to the barycentre of $\nu^f$.

\begin{theorem}\label{lemma dirac}
	Let $f \in \grid{\dom}$ with $\norm{f}_\infty \in \fin$, and let $\nu^f$ be the Young measure that satisfies equality \ref{young equivalence equation}.
	Then $[f] \in L^\infty(\Omega)$ and the following equality holds for a.e.\ $x \in \Omega$:
	\begin{equation}\label{uguaglianza inutile}
	[f](x) = \int_{\R} \tau d\nu^f_x.
	\end{equation}
	Moreover,
	\begin{enumerate}
		\item if $\{f_n\}_{n \in \N}$ is a sequence of $L^\infty$ functions that converges weakly-$\star$ to $\nu^f$ in the sense of equation \ref{young definition}, then $f_n \stackrel{\star}{\weakly} [f]$ in $L^\infty$;
		\item if $\nu^f$ is Dirac, then % for all $f \in C_b(\R)$, $\Phi(\ns{f}(u)) = f(\Phi(u))$ for a.e.\ $x \in \Omega$, i.e.\
		$\nu^f_x$ is the Dirac measure centred at $[f](x)$ for a.e.\ $x \in \Omega$.
	\end{enumerate}
\end{theorem}
\begin{proof}
	Define a function $g$ by posing $g(x) = \int_{\R} \tau d\nu^f_x$ for all $x \in \Omega$.
	Since $|g(x)| \leq \sh{\norm{f}_\infty}$ for a.\ e.\ $x \in \Omega$ and since $\norm{f}_\infty \in \fin$, $g \in L^\infty(\Omega)$.
	By Theorem \ref{young}, for all $\varphi \in C^0_c(\Omega)$ we have the following equalities:
	$$
	\int_{\Omega} g(x) \varphi(x) dx = \int_{\Omega} \int_{\R} \tau d\nu^f_x \varphi(x) dx
	=
	\sh{\langle f, \ns{\varphi} \rangle}
	=
	\int_\Omega [f] \varphi dx.
	$$
	Since $C^0_c(\Omega)$ is dense in $L^1(\Omega)$, we deduce that $g = [f]$ in $L^\infty(\Omega)$, as we wanted.
	%As a consequence, the set $\{x \in \Omega : U(x) \not = \Phi[u](x) \}$ is a null set.
	
	We will now prove (1).
	By hypothesis, from equation \ref{young definition} and from equation \ref{young equivalence equation}, it holds
	$$
	\lim_{n \rightarrow \infty} \int_{\Omega} \Psi(f_n(x)) \varphi(x) dx
	=
	\int_{\Omega} \left(\int_{\R} \Psi d\nu^f_x\right) \varphi(x) dx
	=
	\sh{\langle \ns{\Psi}(f), \ns{\varphi} \rangle}
	$$
	for all $\Psi \in \bcf(\R)$ and for all $\varphi \in C^0_c(\Omega)$.
	As a consequence, by considering a function $\Psi \in \bcf(\R)$ with $\Psi(x) = 1$ for all $x$ satisfying $|x| \leq \sh{\norm{f}_\infty}$, we obtain that the weak-$\star$ limit of the sequence $\{f_n\}_{n \in \N}$ is equal to $[f]$.
	
	Assertion (2) is a consequence of equality \ref{uguaglianza inutile}.
\end{proof}

If the sequence $\{f_n\}_{n \in \N}$ is not bounded in $L^\infty$, then it can be proved that there exists a parametrized measure  $\nu : \Omega \rightarrow \rad(\R)$ such that for all $\Psi \in \bcf(\R)$ the weak-$\star$ limit of the sequence $\{ \Psi(f_n) \}_{n \in \N}$ is the function defined a.e.\ by $\overline{\Psi}(x) = \int_{\R} \Psi d\nu_x$ (for an in-depth discussion of this result, we refer to \cite{ball}).
Notice that $\nu$ takes values in $\rad(\R)$ instead of $\prob(\R)$, since the sequence $\{f_n\}_{n \in \N}$ could diverge in a subset of $\Omega$ with positive measure.

The grid function counterpart of this result is that for any $f \in \grid{\dom}$ there exists a function $\nu^f : \Omega \rightarrow \rad(\R)$ that satisfies equation \ref{young equivalence equation}, even if $\norm{f}_\infty \not \in \fin$.
If $\norm{f}_\infty \not \in \fin$, $\nu_x^f$ might not be a probability measure, but it still satisfies the inequalities $0 \leq \nu_x^f(\R) \leq 1$ for all $x \in \Omega$.
In particular, the difference between $\nu_x^f(\R)$ and $1$ is due to $f$ being unlimited at some non-negligible fraction of $\mu(x)\cap\Lambda^k$.

\begin{theorem}\label{parametrized measures}
	For every $f \in \grid{\dom}$, there exists a parametrized measure $\nu^f : \Omega \rightarrow \rad(\R)$ such that, for all $\Psi \in \bcf(\R)$ and for all $\varphi \in C^0_c(\Omega)$,
	equality \ref{young equivalence equation} holds.
	Moreover, for all $x \in \Omega$ and for all Borel $A \subseteq \R$, $0\leq \nu^f_x(A) \leq 1$.
\end{theorem}
\begin{proof}
	Let $f \in \grid{\dom}$, and for all $n \in \N$ define
	$$f_n(x) = \left\{\begin{array}{ll}
	f(x)& \text{if } |f(x)| \leq n,\\
	n & \text{if } f(x) > n,\\
	-n & \text{if } f(x) < -n.
	\end{array}\right.
	$$
	Since for all $n \in \N$ it holds $\norm{f_n}_\infty \leq n \in \fin$, by Theorem \ref{young} there exists a Young measure $\nu^{n}$ that satisfies
	\begin{equation}\label{you1}
	\sh{\langle \ns{\Psi}(f_n), \ns{\varphi} \rangle}
	=
	\int_{\Omega} \left( \int_{\R} \Psi d \nu^{n}_x \right) \varphi(x) dx.
	\end{equation}
	for all $\Psi \in \bcf(\R)$ and for all $\varphi \in C^0_c(\Omega)$.
	
	Recall that a sequence of parametrized measures $\{\mu^n\}_{n \in \N}$ converges weakly-$\star$ to a parametrized measure $\mu$ if for all $\Psi \in \bcf(\R)$, the sequence $\{\Psi_n\}_{n \in \N}$ of $L^\infty$ functions defined by
	$$
	\Psi_n(x) = \int_{\R} \Psi d \mu^n_x
	$$
	converges weakly-$\star$ to a function $\Psi_{\infty} \in L^\infty(\Omega)$ defined by
	$$
	\Psi_{\infty}(x) = \int_{\R} \Psi d \mu_x.
	$$
	Define $\nu^f$ as the parametrized measure satisfying $\nu^{n} \stackrel{\star}{\weakly} \nu^f$ for some subsequence (not relabelled) of $\{\nu^{n}\}_{n \in \N}$.
	The existence of such a weak-$\star$ limit can be obtained as a consequence of the Banach-Alaouglu theorem %for the proof, ???
	(for further details about the weak-$\star$ limit of measures, we refer to to \cite{evans nonlinear}).
	We claim that $\nu^f$ satisfies equality \ref{young equivalence equation} and that for all $x \in \Omega$, $0\leq \nu^f_x(\R) \leq 1$.
	
	Let $\Psi \in\bcf(\R)$.
	Since $\lim_{|x| \rightarrow \infty} \Psi(x) = 0$, there is an increasing sequence of natural numbers $\{n_i\}_{i \in \N}$ such that if $|x| \geq n_i$, then $|\Psi(x)| \leq 1/i$.
	As a consequence of this inequality, for all $i \in \N$ and for all $\varphi \in C^0_c(\Omega)$ it holds
	$$
	\left| \langle \ns{\Psi}(f_{n_i}), \ns{\varphi} \rangle - \langle \ns{\Psi}(f), \ns{\varphi} \rangle \right| \leq 2/i \norm{\ns{\varphi}}_1.
	$$
	Taking into account equation \ref{you1}, from the previous inequality we obtain
	$$
	\left| \int_{\Omega} \left( \int_{\R} \Psi d \nu^{n_i}_x \right) \varphi(x) dx - \sh{\langle \ns{\Psi}(f), \ns{\varphi} \rangle} \right| \leq 2/i \norm{\varphi}_1.
	$$
	%Since $C^0_c(\Omega)\subseteq L^1(\Omega)$ is dense, 
	As a consequence, we deduce that
	$$
	\lim_{i \rightarrow \infty} \int_{\Omega} \left( \int_{\R} \Psi d \nu^{n_i}_x \right) \varphi(x) dx
	=
	%\lim_{i \rightarrow \infty} \int_{\Omega} \left( \int_{\R} f d \nu^{n_i,-}_x \right) \varphi(x) dx
	%=
	\sh{\langle \ns{\Psi}(f), \ns{\varphi} \rangle}.
	$$
	This is sufficient to entail that $\nu^{n} \stackrel{\star}{\weakly} \nu^f$ and that $\nu^f$ satisfies equality \ref{young equivalence equation}.
	
	The inequality $0 \leq \nu^f_x(A) \leq 1$ for all Borel $A \subseteq \R$ is a consequence of the lower semicontinuity of the weak-$\star$ limit of measures (see for instance theorem 3 of \cite{evans nonlinear}).% inequality $0 \leq \nu^{n,-}_x(\R) \leq 1$ for all $n \in \N$.
\end{proof}

As a consequence of Theorem \ref{parametrized measures}, we deduce that the hypothesis $\norm{f}_\infty \in \fin$ in Theorem \ref{young} can be relaxed.
In particular, if $g$ is a grid function that differs from $f$ at some null set, then $f$ and $g$ induce the same parametrized measure, even if $f \not \equiv g$.

\begin{corollary}\label{corollario importante young}
	Let $L_{N}$ be the Loeb measure obtained from the measure $\mu_N(A) = |A|/N^k$ for all internal $A \subseteq \Lambda^k$.
	If for $f, g \in \grid{\dom}$ it holds $L_{N}(\{x \in \dom : f(x) \not \sim g(x) \}) = 0$, then $\nu^f = \nu^g$.
	If $\norm{f-g}_p \sim 0$, then $\nu^f = \nu^g$.
\end{corollary}
\begin{proof}
	If $L_{N}(\{x \in \dom : f(x) \not \sim g(x) \}) = 0$, then also
	$$L_{N}(\{x \in \dom : \ns{\Psi}(f(x)) \not \sim \ns{\Psi}(g(x)) \}) = 0$$
	for all $\Psi \in \bcf(\R)$.
	This is and the hypothesis $\Psi\in\bcf(\R)$ are sufficient to deduce
	$\langle \ns{\Psi}(f), \ns{\varphi} \rangle \sim \langle \ns{\Psi}(g), \ns{\varphi} \rangle$
	for all $\varphi \in C^0_c(\Omega)$ that, thanks to equation \ref{young equivalence equation}, is equivalent to the equality $\nu^f = \nu^g$.
	
	The hypothesis $\norm{f-g}_p \sim 0$ implies $L_{N}(\{x \in \dom : f(x) \not \sim g(x) \}) = 0$, so the equality between $\nu^f$ and $\nu^g$ is a consequence of the previous part of the proof.
\end{proof}

The above corollary can be seen as the grid function counterpart of Corollary 3.14 of \cite{webbym}, that shows how Young measure ignore concentration phenomena.
We find it useful to discuss this behaviour with an example, that also highlights how a grid function can describe simultaneously very different properties of a sequence of $L^p$ functions.

\begin{example}
	The following example is discussed from the standard viewpoint in \cite{webbym}.
	Consider the sequence $\{f_n\}_{n \in \N}$ defined by $f_n(x) = n \chi_{[1-1/n,1]}$.
	Notice that $\norm{f_n}_\infty = n$, so that the sequence is not bounded in $L^\infty(\R)$.
	For all $\Psi \in \bcf(\R)$ and for all $\varphi \in C^0_c(\R)$, it holds
	$$
	\lim_{n \rightarrow \infty} \int_{\R} \Psi(f_n) \varphi dx = \Psi(0) \int_{\R} \varphi dx
	$$
	so that the sequence $\{f_n\}_{n \in \N}$ converges weakly-$\star$ to the constant Young measure $\nu_x = \delta_0$ for all $x \in \R$.
	
	The sequence  $\{f_n\}_{n \in \N}$ satisfies the $L^1$ uniform bound $\norm{f_n}_1 = 1$ for all $n \in \N$.
	Since for all $\varphi \in \tests(\R)$ it holds
	$$
	\lim_{n \rightarrow \infty} \int_{\R} f_n \varphi dx = \lim_{n \rightarrow \infty} n \int_{[1-1/n,1]} \varphi dx = \varphi(1)
	$$
	the sequence $\{f_n\}_{n \in \N}$ converges in the sense of distributions to $\delta_1$, the Dirac distribution centred at $1$.
	Indeed, it can be proved that the sequence $\{f_n\}_{n \in \N}$ converges weakly-$\star$ to $\delta_1$ in the space $\rad(\R)$ of Radon measures.%, which is the dual of the space $C^0_c(\R)$.
	
	In the setting of grid functions, a representative for the limit of the sequence $\{f_n\}_{n \in \N}$ is given by $f_N = N \chi_{1}$.
	For all $\Psi \in \bcf(\R)$ and for all $\varphi \in C^0_c(\R)$, it holds
	$$
	\langle \ns{\Psi}(f_N) , \ns{\varphi} \rangle = \varepsilon \sum_{x \in \Lambda, \, x \not = 1} \Psi(0) \ns{\varphi}(x) + \varepsilon \ns{\Psi}(N) \varphi(1).
	$$
	Since $\Psi \in \bcf(\R)$, $\ns{\Psi}(N) \sim 0$ and, by Lemma \ref{equivalenza integrali}, we deduce
	$$
	\sh{\langle \ns{\Psi}(f_N) , \ns{\varphi} \rangle} = \Psi(0) \int_{\R} \varphi(x)dx.
	$$
	From the above equality and from equation \ref{young equivalence equation}, we deduce that the Young measure associated to $f_N$ is the constant Young measure $\nu_x = \delta_0$ for all $x \in \R$.	
	Notice that the same result could have been deduced from Corollary \ref{corollario importante young} by noticing that, since $L_{N}(\{x \in \dom : f_N(x) \not \sim 0 \}) = 0$, the Young measure associated to $f_N$ is the same as the Young measure associated to the constant function $c(x) = 0$ for all $x \in \hR$.
	
	As for the distribution $[f_N]$, since for all $\varphi \in \test(\Lambda)$ it holds
	$
	\langle N\chi_1, \varphi \rangle = \varphi(1),
	$
	we deduce that $[f_N] = \delta_1$.
	In particular, the grid function $f_N$ coherently describes the behaviour of the limit of the sequence $\{f_n\}_{n \in \N}$ both in the sense of Young measures and in the sense of distributions.
\end{example}

In the previous example we have considered a grid function $f$ with $\norm{f}_1 \in \fin$, and we verified that the parametrized measure associated to $f$ was indeed a Young measure.
This result holds under the more general hypothesis that $\norm{f}_p \in \fin$.

\begin{proposition}\label{proposition young lp}
	If $\norm{f}_p \in \fin$, then $\nu^f_x$ is a probability measure for a.e.\ $x \in \Omega$.
\end{proposition}
\begin{proof}
	If for some $x \in \Omega$ it holds $\nu^f_x(\R) < 1$, then there exists $y \in \dom$, $y \sim x$ such that $f(y) \not \in \fin$.
	The hypothesis $\norm{f}_p \in \fin$ implies $L_{N}(\{y \in \dom : f(y) \not \in \fin \}) = 0$: this is sufficient to conclude that $\leb(\{x \in \Omega : \nu^f_x(\R) < 1 \}) = 0$, as desired.
\end{proof}

We will conclude the discussion of the relations between grid functions and parametrized measures by determining the parametrized measure associated to a periodic grid function with an infinitesimal period.
This is the grid function counterpart of the formula for the Young measure associated to the limit of a sequence of periodic functions (see Example 3.5 of \cite{balder}).
We will prove this result for $k = 1$, as the generalization to an arbitrary dimension is mostly a matter of notation.

\begin{proposition}\label{homogeneous}
	If $f \in \grid{\Lambda}$ is periodic of period $M\varepsilon \sim 0$, then the parametrized measure $\nu$ associated to $f$ is constant, and
	$$
	\int_{\R} \Psi d\nu_x = \sh{\left( \frac{1}{M} \sum_{i = 0}^{M-1} \ns{\Psi}(f(i\varepsilon))\right)}
	$$
	for all $x \in \Omega$ and for all $\Psi \in \bcf(\R)$.
	%In particular, if $A \subseteq \R$
	%$$
	%\nu(x)(A) = \sh{\left(\frac{\loeb{\{i : \sh{u(i\varepsilon)} \in A, \ 0 \leq i \leq M-1\}}}{M} \right)}.
	%$$
\end{proposition}
\begin{proof}
	Without loss of generality, let $M \in \ns{\N}$ and let $f$ be periodic over $[0,(M-1)\varepsilon]\cap\Lambda$, with $M\varepsilon\sim0$.	
	
	Let $\Psi \in \bcf(\R)$.
	At first, we will prove that $\frac{1}{M} \sum_{i = 0}^{M-1} \ns{\Psi}(f(i\varepsilon))$ is finite: in fact,
	\begin{equation}\label{finiteness}
	\inf_{x\in\hR}\ns{\Psi}(x) \leq
	\frac{1}{M} \sum_{i = 0}^{M-1} \ns{\Psi}(f(i\varepsilon))
	\leq \sup_{x \in \hR} \ns{\Psi}(x)
	\end{equation}
	and by the boundedness of $\Psi$, we deduce that $\frac{1}{M} \sum_{i = 0}^{M-1} \ns{\Psi}(f(i\varepsilon))$ is finite.
	
	Let now $\varphi \in S^0(\Lambda)$ with $\supp \varphi \subset [a,b]$, $a, b \in \fin$.
	Then there exists $h,k \in \ns{\N}$ satisfying $a \sim Mh\varepsilon$ and $b \sim Mk\varepsilon$.
	We have the equalities
	\begin{eqnarray}\notag
	\langle \ns{\Psi}(f), \varphi \rangle
	&\sim&
	\varepsilon \sum_{x\in\ni{Mh\varepsilon}{Mk\varepsilon}} \Psi(f(x))\varphi(x)\\\notag
	&=&
	\varepsilon \sum_{j = h}^{k} \left(\sum_{i = 0}^{M-1} \ns{\Psi}(f(i\varepsilon)) \varphi(jM\varepsilon + i\varepsilon) \right)\\\notag
	&=&
	\varepsilon \sum_{j = h}^{k} \left( \left(\sum_{i = 0}^{M-1} \ns{\Psi}(f(i\varepsilon))\right) (\varphi(jM\varepsilon) + e(j)) \right)\\\label{equiv misura 3}
	&=&
	\left(\frac{1}{M} \sum_{i = 0}^{M-1} \ns{\Psi}(f(i\varepsilon))\right) \left(M\varepsilon \sum_{j = h}^{k} ( \varphi(jM\varepsilon) + e(j))\right).
	\end{eqnarray}
	Let
	$$
	e = \max_{0 \leq i \leq M,\ k\leq j \leq h}\{|\varphi(jM\varepsilon)-\varphi(jM\varepsilon+i\varepsilon)|\}.
	$$
	Since $\varphi \in S^0(\Lambda)$ and $\supp \varphi \subset \fin$, $e \sim 0$ and, as a consequence, $|e(j)| \leq e \sim 0$.
	We deduce
	$$\left|M\varepsilon \sum_{j = k}^{h} e(j)\right| \leq M\varepsilon (k-h) e \sim (b-a)e \sim 0$$
	%We claim that $\varepsilon (k-h)$ is finite: otherwise, $M \varepsilon (k-h) \sim b-a$ would be infinite, against the fact that $\varphi \in \test(\Lambda)$.
	and, by equation \ref{finiteness},
	\begin{equation}\label{equiv misura 4}
	\left( \frac{1}{M} \sum_{i = 0}^{M-1} \ns{\Psi}(f(i\varepsilon))\right) \left( M\varepsilon \sum_{j = h}^{k} e(j)\right) \sim 0.
	\end{equation}
	Since $M\varepsilon \sim 0$,
	\begin{equation}\label{equiv misura 5}
	M\varepsilon \sum_{j = h}^{k} (\varphi(jM\varepsilon)
	\sim
	\int_{\sh{a}}^{\sh{b}} \sh{\varphi(x)} dx.
	\end{equation}
	Putting together equalities \ref{equiv misura 3}, \ref{equiv misura 4} and \ref{equiv misura 5}, we conclude
	$$
	\sh{\langle \ns{\Psi}(f), \varphi(x) \rangle}
	=
	\sh{\left(\frac{1}{M} \sum_{i = 0}^{M-1} \ns{\Psi}(f(i\varepsilon))\right)} \int_{\sh{a}}^{\sh{b}} \sh{\varphi(x)} dx
	$$
	as we wanted.
\end{proof}

\section{The grid function formulation of partial differential equations}\label{solutions}

We have seen that the space of grid functions extends coherently both the space of distributions and the space of Young measures.
For this reason, we believe they can successfully applied to the study of partial differential equations.
In this section, we will give some results that allow to give a coherent grid function formulation of stationary and time-dependent PDEs.
A solution to the grid function formulation can then be used to define a standard solution to the original problem.
% that, if the solutions to the grid function formulation are regular enough, they induce standard solutions to the original problem.
%Moreover, the existence of regular solutions to the original problem is equivalent to the existence of regular solutions to the grid function formulation.
%If the original problem does not have solutions in the sense of distributions, then we regard the solution to the grid function formulation as a generalized solution.
It turns out that, for some nonlinear problems, this process will give rise to measure-valued solutions.

\subsection{The grid function formulation of linear PDEs}\label{linear pdes}

A linear PDE can be written in the most general form as
\begin{equation}\label{lift2}
L(u) = T,
\end{equation}
with $T \in \tests'(\Omega)$, where $L : \tests'(\Omega) \rightarrow \tests'(\Omega)$ is linear, and where the equality is meant in the sense of distributions, i.e.\
$$
\ldual L(u), \varphi \rdual = \ldual T, \varphi \rdual
$$
for all $\varphi \in \test(\Omega)$.
We would like to turn problem \ref{lift2} in a problem in the sense of grid functions, i.e.
\begin{equation}\label{lift1}
\L(u) = \f,
\end{equation}
with $\f \in \B$, where $\L : \B \rightarrow \B$ is $\hR$-linear, and where the equality is pointwise equality.
We would like to determine sufficient conditions that ensure equivalence between problem \ref{lift1} and problem \ref{lift2}, in the sense that \ref{lift2} has a solution if and only if \ref{lift1} has a solution.

Such a coherent formulation of linear PDEs relies upon the existence of $\hR$-linear extensions of linear functionals over the space of distributions.
Recall that every linear functional $L : \tests'(\Omega) \rightarrow \tests'(\Omega)$ induces an adjoint $\M : \test(\Omega) \rightarrow \test(\Omega)$ that satisfies
$$
\ldual L(T), \varphi \rdual = \ldual T, \M (\varphi) \rdual
$$
for all $T \in \tests'(\Omega)$ and for all $\varphi \in \tests(\Omega)$.
If we find a $\hR$-linear extension of $\M$ in the sense of grid functions, by taking the adjoint we are able to define a $\hR$-linear extension of $L$.

\begin{lemma}\label{lemma lift duale}
	For every 
	linear $L: \tests(\Omega) \rightarrow \tests(\Omega)$ there is a $\hR$-linear $\L :\grid{\dom} \rightarrow \grid{\dom}$ such that $\L(\ns{\varphi}) = \ns(L(\varphi))_{|\dom}$ for all $\varphi \in \tests(\Omega)$.
\end{lemma}
\begin{proof}
	For $\varphi \in \tests'(\Omega)$ define
	$$
	U(\varphi) = \{ \L : \grid{\dom} \rightarrow \grid{\dom} \text{ such that } \L  \text{ is } \hR\text{-linear and } \L(\ns{\varphi}) = \ns{L(\varphi)}_{|\dom} \}
	$$
	and let $U = \{ U(\varphi) : \varphi \in \tests(\Omega)\}$.
	If we prove that $U$ has the finite intersection property, then, by saturation, $\bigcap U \not = \emptyset$, and any $\L \in \bigcap U$ is a $\hR$-linear function that satisfies $\L(\ns{\varphi}) = \ns(L(\varphi))_{\Lambda}$ for all $\varphi \in \tests(\Omega)$.
	
	We will prove that, if $\varphi_1, \ldots, \varphi_n \in \tests$, then $\bigcap_{i = 1}^n U(\varphi_i) \not = \emptyset$ by induction over $n$.
	If $n = 1$, we need to show that $U(\varphi) \not = \emptyset$ for all $\varphi \in \tests'(\Omega)$.
	If $\varphi = 0$, then the constant function $\L (f) = 0$ for all $f \in \grid{\dom}$ belongs to $U(\varphi)$.
	If $\varphi \not = 0$, let $f = \ns{\varphi}_{|\dom}$, $g = \ns{(L(\varphi)})_{|\dom}$, and let $\{f, b_2, \ldots, b_M\}$ be a $\ns{}$basis of $\grid{\dom}$.
	Define also
	$$
	\L\left(a_1f + \sum_{i = 2}^M a_i b_i \right) = a_1g
	$$
	for all $a_1, \ldots a_M \in \hR$.
	By definition, $\L$ is $\hR$-linear and $\L \in U(\varphi)$.
	
	We will now show that if $\bigcap_{i = 1}^{n-1} U(\varphi_i) \not = \emptyset$ for any choice of $\varphi_1, \ldots, \varphi_{n-1} \in \tests(\Omega)$, then also $\bigcap_{i = 1}^{n} U(\varphi_i) \not = \emptyset$ for any choice of $\varphi_1, \ldots, \varphi_{n} \in \tests(\Omega)$.
	If $\{\varphi_1, \ldots, \varphi_n\}$ are linearly dependent, %then $T_n = \sum_{i=1}^{n-1} r_i T_i$ with $r_i \in \R$.
	%As a consequence, if $f_1, \ldots, f_{n-1}$ satisfy $\Phi[f_i] = T_i$, then $f_n = \sum_{i=1}^{n-1} r_i f_i \in \B$ and satisfies $\Phi[f_n] = T_n$.
	thanks to linearity of $L$, any $\L\in \bigcap_{i = 1}^{n-1} U(\varphi_i)$ satisfies $\L\in \bigcap_{i = 1}^n U(\varphi_i)$. %, implying that $\bigcap_{i = 1}^n U(T_i) \not = \emptyset$.
	If $\{\varphi_1, \ldots, \varphi_n\}$ are linearly independent, let $f_n = (\ns{\varphi_n})_{|\dom}$, let $g_n = \ns{(L(\varphi_n)})_{|\dom}$ and let $\{f_n, b_{2}, \ldots, b_M\}$ be a $\ns{}$basis of $\grid{\dom}$.
	For any $\L\in \bigcap_{i = 1}^{n-1} U(\varphi_i)$, define $\overline{\L} : \grid{\dom} \rightarrow \grid{\dom}$ by
	$$
	\overline{\L}\left(a_1f_n + \sum_{i = 2}^M a_i b_i \right) =
	a_1 g_n + \L\left(\sum_{i = 2}^M a_i b_i\right)
	$$
	for all $a_1, \ldots a_M \in \hR$.
	Then $\overline{\L}\in \bigcap_{i = 1}^n U(\varphi_i)$.
	This concludes the proof.
\end{proof}

\begin{theorem}\label{teorema lift lineare}
	For every 
	linear $L: \tests'(\Omega) \rightarrow \tests'(\Omega)$ there is a $\hR$-linear $\L :\grid{\dom} \rightarrow \grid{\dom}$ such that 
	$\sh{\langle \L (f), \ns{\varphi} \rangle} = \ldual L [f], \varphi\rdual$ for all $f \in \test'(\dom)$ and for all $\varphi \in \tests(\Omega)$.
	As a consequence, the following diagram commutes:
	\begin{equation}\label{lift lineare}
	\begin{array}{ccc}
	\B & \stackrel{\L}{\longrightarrow} & \B \\
	\Phi \circ \pi \downarrow & & \downarrow \Phi \circ \pi\\
	\tests'(\Omega) & \stackrel{L}{\longrightarrow} & \tests'(\Omega).
	\end{array}
	\end{equation}
\end{theorem}
\begin{proof}	
	Let $\M$ be the adjoint of $L$, and let $\M_\Lambda$ be the $\hR$-linear operator coherent with $\M$ in the sense of Lemma \ref{lemma lift duale}.
	Define $\langle \L(f), \varphi \rangle = \langle f, \M_\Lambda(\varphi)\rangle$ for all $\varphi \in \test(\dom)$.
	From this definition, $\hR$-linearity of $\L$ can be deduced from the $\hR$-linearity of $\M_\Lambda$.
	
	We will now prove that $\L$ satisfies $\sh{\langle \L (f), \ns{\varphi} \rangle} = \ldual L [f], \varphi\rdual$ for all $f \in \test'(\dom)$ and for all $\varphi \in \tests(\Omega)$.
	Let $f \in \test'(\dom)$: for any $\varphi \in \tests(\Omega)$, thanks to Lemma \ref{lemma lift duale} we have the equalities
	$$
	\langle \L(f), \ns{\varphi} \rangle
	=
	\langle f, \M_\Lambda(\ns{\varphi}_{|\dom})\rangle
	=
	\langle f, \ns{\M(\varphi)}\rangle
	\sim
	\ldual [f], \M(\varphi)\rdual
	=
	\ldual L[f], \varphi\rdual,
	$$
	as we wanted.
	
	By Lemma \ref{character} and thanks to the previous equality, if $f \in \B$, then $\L(f) \in \B$. %, so that $\sh{\langle \L(f), \varphi \rangle} = \ldual L[f], \sh{\varphi} \rdual$ for all $\varphi \in \test(\dom)$.
	This concludes the proof of the commutativity of diagram \ref{lift lineare}.
\end{proof}

From the previous Theorem, we obtain some sufficient conditions that ensure the equivalence between the linear problem \ref{lift2} in the sense of distributions and the linear problem \ref{lift1} in the sense of grid functions.

\begin{theorem}\label{theorem lift lineare}
	Let $L : \tests'(\Omega) \rightarrow \tests'(\Omega)$ be linear, and let $\L : \grid{\dom} \rightarrow \grid{\dom}$ any function such that diagram \ref{lift lineare} commutes.
	Let also $T \in \tests'(\Omega)$.
	Then problem	
	\ref{lift2}
	has a solution if and only if  problem
	\ref{lift1}
	has a solution $u \in \B$ for some $\f$ satisfying $[\f] = T$.
\end{theorem}
\begin{proof}
	By Theorem \ref{teorema lift lineare}, if problem \ref{lift1} has a solution $u$, then $[u]$ satisfies problem \ref{lift2}.
	
	The other implication is a consequence of Theorem \ref{teorema lift lineare} and of surjectivity of $\Phi$: suppose that \ref{lift2} has a solution $v$.
	The commutativity of diagram \ref{lift lineare} ensures that for any $u \in \Phi^{-1}(v)$ it holds $[\L (u)] = T$, hence for $\f = \L(u)$ problem \ref{lift1} has a solution.
\end{proof}

Thanks to this equivalence result, any linear PDE can be studied in the setting of grid functions with the techniques from linear algebra.

As an example of the grid function formulation of a linear PDE, we find it useful to discuss the Dirichlet problem.

\begin{definition}\label{def dirichlet}
	Let $\Omega \subset \R^k$ be open and bounded, $h \in \N$, $a_{\alpha, \beta} \in C^\infty(\Omega)$, and let
	$$
	L(v) = \sum_{0\leq|\alpha|,|\beta|\leq h} (-1)^{|\alpha|}D^{\alpha}(a_{\alpha, \beta}D^\beta v).
	$$
	The Dirichlet problem is the problem of finding $v$ satisfying
	\begin{equation}\label{dirichlet}
	\left\{
	\begin{array}{l}
	L(v) = f \text{ in } \Omega\\
	D^{\alpha}u = 0 \text{ for } |\alpha|\leq h-1 \text{ in } \partial\Omega.
	\end{array}\right.
	\end{equation}
	If $f \in C_b(\Omega)$, then $v$ is a classical solution of the Dirichlet problem if
	\begin{equation}\label{classi dirichlet}
	v \in C^{2h}_b(\Omega) \cap C^{2h-1}_b(\overline{\Omega}) \text{ and } L(v) = f.
	\end{equation}
	If $f \in L^2(\Omega)$, then $v$ is a strong solution of the Dirichlet problem if
	$$
	v \in H^{2h}(\Omega) \cap H^{h}_0(\overline{\Omega}) \text{ and } L(v) = f \text{ a.e.}
	$$
	%Notice that, in this case $L(v) = f$ is an equality in $L^2(\Omega)$, i.e.\ is an equality almost everywhere.
	If $f \in H^{-h}(\Omega)$, then $v$ is a weak solution of the Dirichlet problem if
	\begin{equation}\label{weak dirichlet}
	v \in H^{h}_0(\Omega) \text{ and }
	\sum_{0\leq|\alpha|,|\beta|\leq h} \int a_{\alpha, \beta}D^\beta v D^\alpha w = f(w)
	\text{ for all } w \in H^h_0(\Omega).
	\end{equation}
	
	\begin{comment}
	\color{blue}
	If $0 < h \leq k$, $f \in H^{-k-h}(\Omega)$, then $v$ is a very weak solution of the Dirichlet problem if
	$$
	v \in H^{k}_0(\Omega) \text{ and }
	\sum_{\substack{0\leq|\alpha|\leq k,\\ 0 \leq |\beta| \leq h,\\ 0 \leq |\gamma| \leq k-h}}
	(-1)^{|\beta|}\int D^\gamma v D^{\beta}(a_{\alpha, \beta}D^\alpha w) = f(w)
	\text{ for all } w \in H^{k+h}_0(\Omega).
	$$
	\end{comment}
\end{definition}

\begin{definition}
	A grid function formulation of the Dirichlet problem \ref{dirichlet} is the following: let
	$$
	\L(u) = \sum_{0\leq|\alpha|,|\beta|\leq h} (-1)^{|\alpha|}\D^{\alpha}(\ns{a}_{\alpha, \beta}\D^\beta u).
	$$
	The Dirichlet problem is the problem of finding $u \in \grid{\dom}$ satisfying
	\begin{equation}\label{grid dirichlet}
	\left\{
	\begin{array}{l}
	\L(u) = P(\ns{f}) \text{ in } \dom\\
	\D^{\alpha}u = 0 \text{ in } \partial_\Lambda^\alpha\dom\text{ for } |\alpha|\leq s-1.
	\end{array}\right.
	\end{equation}
\end{definition}

Notice that equation \ref{grid dirichlet} is satisfied in the sense of grid functions, i.e.\ pointwise, while equation \ref{dirichlet} assumes the different meanings shown in Definition \ref{def dirichlet}.

A priori, a solution $u$ of problem \ref{lift1} induces a solution $[u]$ of problem \ref{lift2} in the sense of distributions.
However, if $[u]$ is more regular, it is a solution to \ref{lift2} in a stronger sense.
%If $u$ is regular enough, then $[u]$ is a solution of the Dirichlet problem in one of the senses of definition \ref{def dirichlet}.

\begin{theorem}\label{teorema dirichlet}
	Let $u$ be a solution of problem \ref{grid dirichlet}.
	Then
	\begin{enumerate}
		\item if $f \in C_b(\Omega)$ and $[u] \in C^{2h}_b(\Omega) \cap C^{2h-1}_b(\overline{\Omega})$, then $[u]$ is a classical solution of the Dirichlet problem;
		\item if $f \in L^2(\Omega)$ and $[u] \in H^{2h}(\Omega) \cap H^{h}_0(\Omega)$, then $[u]$ is a strong solution of the Dirichlet problem;
		\item if $f \in H^{-h}(\Omega)$ and $[u] \in H^{h}_0(\Omega)$, then $[u]$ is a weak solution of the Dirichlet problem, i.e.\ $[u]$ satisfies \ref{weak dirichlet}.		
	\end{enumerate}
\end{theorem}
\begin{proof}
	A solution $u$ of problem \ref{grid dirichlet} satisfies the equality
	\begin{eqnarray*}
		\langle P(\ns{f}), \varphi \rangle
		&=& \sum_{0\leq|\alpha|,|\beta|\leq h} (-1)^{|\alpha|} \langle \D^{\alpha}(\ns{a}_{\alpha, \beta}\D^\beta u), \varphi \rangle \\
		&=& \sum_{0\leq|\alpha|,|\beta|\leq h} \langle \ns{a_{\alpha, \beta}} \D^\beta u, \D^{\alpha}\varphi \rangle.
		%\label{grid vweak}
		%&=& \sum_{0\leq|\alpha|,|\beta|\leq k} (-1)^{|\beta|} \langle u, \D^\beta(\ns{a_{\alpha, \beta}} \D^{\alpha}\phi) \rangle.
	\end{eqnarray*}
	for all $\varphi \in \test'(\dom)$.
	
	We will now prove (1).
	If $f \in C_b(\Omega)$, then by Lemma \ref{proiez s-c}, $[P(\ns{f})]=f$.
	By Theorem \ref{teorema equivalenza derivate2}, $[\D^{\beta}u] = D^{\beta}[u]$, and $[\ns{a}_{\alpha, \beta}\D^\beta u] = a_{\alpha, \beta} D^{\beta} [u]$, so that
	$$
	[(-1)^{|\alpha|}\D^{\alpha}(\ns{a}_{\alpha, \beta}\D^\beta u)]
	=
	(-1)^{|\alpha|}D^{\alpha}(a_{\alpha, \beta}D^\beta [u]).
	$$
	We deduce that $[u]$ satisfies equation \ref{classi dirichlet} in the classical sense, as desired.
	
	The proof of parts (2) and (3) is similar to that of part (1).
	The only difference is that it relies on Lemma \ref{corollario lp} instead of Lemma \ref{proiez s-c}.
	%\color{blue}
	%L'equazione \ref{grid vweak} serve per le soluzioni molto deboli.
\end{proof}

\begin{remark}
	While Theorem \ref{teorema lift lineare} and Theorem \ref{theorem lift lineare} do not explicitly determine an extension $\L$ for a given linear PDE, they determine a sufficient condition for problem \ref{lift1} to be a coherent representation of problem \ref{lift2} in the sense of grid function.
	In the practice, an explicit extension $\L$ of a linear $L : \tests'(\Omega) \rightarrow \tests'(\Omega)$ can be determined from $L$ by taking into account that
	\begin{itemize}
		\item thanks to Theorem \ref{teorema equivalenza derivate2}, derivatives can be replaced by finite difference operators;
		\item shifts can be represented in accord to Corollary \ref{corollario shift};
		\item if $a \in C^\infty(\Omega)$, then $[\ns{a}f]=a[f]$ for all $f \in \B$, since for all $\varphi \in \test(\dom)$, $\ns{a}\varphi \in \test(\dom)$, and we have the equalities
		$$
		\sh{\langle \ns{a} f, \varphi \rangle} = \sh{\langle f, \ns{a}\varphi \rangle} = \ldual [f], a\varphi\rdual = \ldual a[f], \varphi\rdual.
		$$
	\end{itemize}
	
	Similarly, we have not established a canonical representative $\f$ for $T$.
	%If $f$ is regular enough, then $\f$ can be chosen as $\f = \ns{f}_{|\dom}$, while if $f \in \tests'(\Omega)$ then we could take $\f = l(f)$, where $l$ has been defined in Theorem \ref{schwartz}.
	However, observe that for all $g \in \grid{\dom}$ and for all $x \in \dom$ it holds
	$$
	g(x) = \sum_{y \in \dom} g(y) N^k \chi_{y}(x)
	$$
	Moreover, $\chi_{y}(x) = \chi_0(x-y)$, so that once a solution $u_0$ for the problem $\L u = N^k\chi_0$ is determined, a solution for $\L (u) = g$ can be determined from the above equality by posing
	\begin{equation}\label{convo}
	u_{g}(x) = \sum_{y \in \dom} g(y) u_0(x-y).
	\end{equation}
	In fact, by linearity of $\L$ we have that, for all $x \in \dom$,
	\begin{eqnarray*}
		\L(u_g(x)) &=& \displaystyle \L \left( \sum_{y \in \dom} g(y) u_0(x-y) \right)\\ \\
		& = & \displaystyle \sum_{y \in \dom} g(y) \L(u_0(x-y))\\
		& = & \displaystyle\sum_{y \in \dom} g(y) N^k \chi_{0}(x-y)\\
		& = & g(x).
	\end{eqnarray*}
	In particular, $u_0$ plays the role of a fundamental solution for problem \ref{lift1}, while equality \ref{convo} can be interpreted as the discrete convolution between $g$ and $u_0$.
	As a consequence, the study of a linear problem \ref{lift1} can be carried out by determining the solutions to the problem $\L (u) = N^k\chi_0$.
\end{remark}

\subsection{The grid function formulation of nonlinear PDEs}\label{section nonli}

A nonlinear PDE can be written in the most general form as
$$%\begin{equation}\label{nonli}
F(u) = f,
$$%\end{equation}
usually with $u \in V \subseteq L^2(\Omega)$ and $F: V \rightarrow W \subseteq L^2(\Omega)$.
As in the linear case, the grid function formulation of nonlinear problems is based upon the possibility to coherently extend every continuous $F : L^2(\Omega) \rightarrow L^2(\Omega)$ to all of $\grid{\dom}$.
Since the proofs of the following theorems are based upon Lemma \ref{norma r}, we will impose the additional hypothesis that the Lebesgue measure of $\Omega$ is finite.
Notice that, in contrast to what happened for Theorem \ref{teorema lift lineare}, in the proof of Theorem \ref{teorema lifting nonlineari}, we will be able to explicitly determine a particular extension $\F$ for a given continuous $F : L^2(\Omega) \rightarrow L^2(\Omega)$.

\begin{theorem}\label{teorema lifting nonlineari}
	Let $\leb(\Omega) < +\infty$ and let $F : L^2(\Omega) \rightarrow L^2(\Omega)$ be continuous.
	Then there is a function $\F : \grid{\dom} \rightarrow \grid{\dom}$ that satisfies
	\begin{enumerate}
		\item whenever $u, v \in \grid{\dom}$ are nearstandard in $L^2(\Omega)$, $\norm{u-v}_2 \sim 0$ implies $\norm{\F(u)-\F(v)}_2 \sim 0$;
		\item for all $f \in L^2(\Omega)$,
		$[\F(\lpi(\ns{f}))] = F(f)$.
	\end{enumerate}
\end{theorem}
\begin{proof}
	We will show that the function defined by
	$\F(u) = \lpi(\ns{F(\lp{u})})$ for all  $u \in \grid{\dom}$ satisfies the thesis.
	By continuity of $F$, whenever $u$ and $v$ are nearstandard in $L^2(\Omega)$ we have
	$$
	\norm{u-v}_2 \sim 0 \text{ implies } \norm{\ns{F(u)}-\ns{F(v)}}_2 \sim 0,
	$$
	and, by Lemma \ref{norma r},
	$$
	\norm{\ns{F(u)}-\ns{F(v)}}_2 \sim 0 \text{ implies } \norm{\F(u)-\F(v)}_2 \sim 0,
	$$
	hence (1) is proved. %$\F$ is continuous in the sense of the $\L^2$ norm.
	
	We will now prove that $[\F(\lpi(\ns{f}))] = F(f)$.
	By Lemma \ref{norma r}, we have $\norm{\ns{f}-\lpi(\ns{f})}_2 \sim 0$
	%This is sufficient to conclude that $\F(\ns{f})$ is nearstandard in $L^2(\Omega)$.
	and, by continuity of $\ns{F}$, $\norm{\ns{F}(\ns{f})-\ns{F}(\lpi(\ns{f}))}_2 \sim 0$.
	From Lemma \ref{questo corollario} we have $[\F(\ns{f})] = [\lpi(\ns{F(\ns{f})})] = F(f)$, as desired.
\end{proof}

\begin{remark}\label{remark 1}
	In the same spirit, if $F: V \rightarrow W$ is continuous and the space of grid functions can be continuously embedded in $\ns{V}$ and $\ns{W}$, then one can prove similar theorems by varying condition (1) in order to properly represent the topologies on the domain and the range of $F$.
	For instance, if $F: H^1(\Omega) \rightarrow L^2(\Omega)$, then (1) would be replaced by
	$$\norm{u-v}_{H^1} \sim 0 \text{ implies } \norm{\F(u)-\F(v)}_2 \sim 0,$$
	where $\norm{u-v}_{H^1}$ is defined in the expected way as
	$$\norm{u-v}_{H^1} = \norm{u-v}_{2}+\norm{\grad (u-v)}_2.$$
\end{remark}

Condition (1) of Theorem \ref{teorema lifting nonlineari} is a continuity requirement for $\F$, and condition (2) implies coherence of $\F$ with the original function $F$, so that
theorem \ref{teorema lifting nonlineari} ensures that for all continuous $F : L^2 \rightarrow L^2$ there is a function $\F:\grid{\dom}\rightarrow \grid{\dom}$ which is continuous and coherent with $F$.
%The grid function formulation of partial differential equations is grounded upon this result.
This result allows to formulate nonlinear PDEs in the setting of grid functions.

\begin{theorem}\label{nonlinear pdes}
	Let $\leb(\Omega) < +\infty$, let $F:L^2(\Omega)\rightarrow L^2(\Omega)$ and let $\F : \grid{\dom} \rightarrow \grid{\dom}$ satisfy conditions (1) and (2) of Theorem \ref{teorema lifting nonlineari}.
	Let also $f \in L^2(\Omega)$.
	Then the problem of finding $v \in L^2(\Omega)$ satisfying
	\begin{equation}\label{eq nonli}
	F(v) = f
	\end{equation}
	has a solution if and only if there exists a solution $u \in \grid{\dom}$, $u$ nearstandard in $L^2(\Omega)$, that satisfy
	\begin{equation}\label{grid nonli}
	\F(u) = f_\Lambda
	\end{equation}
	for some $f_\Lambda\in\grid{\dom}$ with $[f_\Lambda]=f$, and in particular for $f_\Lambda = P(\ns{f})$.
\end{theorem}
\begin{proof}
	Suppose that \ref{grid nonli} with $f_\Lambda = \lpi(\ns{f})$ has a solution $u$.
	Since $[\lpi(\ns{f})] = f$ by Corollary \ref{questo corollario}, $u$ satisfies the equality
	$[\F(u)] = f$
	in the sense of distributions.
	At this point, if $u$ is nearstandard in $L^2(\Omega)$, by Lemma \ref{norma r} we have $\norm{\ns{[u]} - u}_2 \sim 0$, so that $[u] \in L^2(\Omega)$, and condition (2) of Theorem \ref{teorema lifting nonlineari} ensures that $[\F(u)] = F([u])$, so that $[u]$ is a solution of \ref{eq nonli}.
	
	For the other implication, suppose that $v$ is a solution to \ref{eq nonli}.
	Then, by condition (2) of Theorem \ref{teorema lifting nonlineari}, $[\F(\lpi(\ns{v}))] = F(v)=f$, so that problem \ref{grid nonli} has a solution.
\end{proof}

If $u$ is a solution to \ref{grid nonli} but it is not nearstandard in $L^2(\Omega)$, i.e.\ if $\norm{\ns{[u]} - u}_2 \not \sim 0$, $[\F(u)]$ needs not be equal to $F([u])$.
In fact, if $[u] \in L^2(\Omega)$ and $\norm{\ns{[u]} - u}_2 \not \sim 0$, we have argued in Section \ref{sub lp} that we expect $u$ to feature either strong oscillations or concentrations.
Due to these irregularities, we have no reasons to expect that $[\F(u)](x)$, that represents the mean of the values assumed by $\F(u)$ at points infinitely close to $x$, is related to $F([u])(x)$, that represents the function $F$ applied to the mean of the values assumed by $u$ at points infinitely close to $x$.
However, as we have seen in Section \ref{young}, if $\norm{u}_\infty \in \fin$, then $u$ can be interpreted as a Young measure $\nu^u$.
If the composition $F(\nu^u)$ is defined in the sense of equation \ref{young equivalence equation}, then $\nu^u$ satisfies
$$
\int_{\Omega} \int_{\R} F(\tau) d\nu^u(x) \varphi(x) dx = \sh{ \langle \F(u), \varphi \rangle}
= \sh{\langle \lpi{(\ns{f})}, \varphi\rangle}
=
\int_{\Omega} f\varphi dx
$$
for all $\phi \in \tests'(\Omega)$, and can be regarded as a Young measure solution to equation \ref{eq nonli}.
%The meaning of $\nu^u$ as a generalized solution of equation \ref{nonli}, need to be addressed on a case-by-case basis.
%We will discuss a specific nonlinear problem and the meaning of its generalized solutions in Section \ref{illposed}.
In particular, since Young measures describe weak-$\star$ limits of sequences of $L^\infty$ functions, the relation between $F(\nu^u)$ and problem \ref{eq nonli} is the following: there exists a family of regularized problems
$$
F_\eta(u) = f_\eta
$$
and a family $\{u_\eta\}_{\eta > 0}$ of $L^2(\Omega)\cap L^\infty(\Omega)$ solutions of these problems such that $\nu^u$ represents the weak-$\star$ limit of a subsequence of $\{u_\eta\}_{\eta > 0}$, and $F(\nu^u)$ is the corresponding weak limit of the sequence $\{F(u_\eta)\}_{\eta > 0}$.

In the case that $\norm{u}_\infty$ is infinite or that $[u] \not \in L^2(\Omega)$, we consider $u$ as a generalized solution of problem \ref{eq nonli} in the sense of grid functions.
Moreover, we expect $u$ to capture both the oscillations and the concentrations we would expect from a sequence of solutions of some family of regularized problems of \ref{eq nonli}.
A more in-depth example of this behaviour is discussed in the grid function formulation of a class of ill-posed PDEs in \cite{illposed}.

\begin{remark}\label{remark 2}
	Notice that if $\F$ satisfies the stronger continuity hypothesis
	\begin{equation}\label{weak continuity}
	u \equiv v \text{ implies } \F(u) \equiv \F(v),
	\end{equation}
	then $\F$ has a standard part $\tilde{F}$ defined by
	$$
	\tilde{F}(g) = [\F(P(\ns{g}))]
	$$
	for any $g \in L^2(\Omega)$.
	Moreover, from Lemma \ref{norma r} and from Theorem \ref{teorema lifting nonlineari}, we deduce that $\tilde{F} = F$.
	As a consequence, any grid function $u$ that satisfies $\F(u) = P(\ns{f})$ induces a solution to problem \ref{eq nonli}.
	
	However, the continuity condition \ref{weak continuity} holds only for very regular functions, and it fails for many of the %very regular functions, such as $F(u) = u^2$,
	functions that still satisfy the hypotheses of Theorem \ref{teorema lifting nonlineari}.
\end{remark}

\begin{remark}
	If the function $F$ appearing in equation \ref{eq nonli} can be expressed as
	$$
	F = L \circ G,
	$$
	where $G$ is nonlinear and $L$ is linear,
	the equivalence between the standard notions of solutions for the PDE \ref{eq nonli} and one of its formulations in the sense of grid functions can be obtained by a suitable combination of the results of Theorem \ref{theorem lift lineare} and of Theorem \ref{nonlinear pdes}.
\end{remark}

\subsection{Time dependent PDEs}

Time dependent PDEs have been studied in the setting of nonstandard analysis by a variety of means.
A possibility is to give a nonstandard representation of a given time dependent PDE by discretizing in time as well as in space, and by defining a standard solution to the original problem by the technique of stroboscopy.
%it is well-known that the time step for the discretization in space is subject to restrictions that vary from problem to problem.
In \cite{imme1}, van den Berg showed how the stroboscopy technique can be extended to the study of a class of partial differential equations of the first and the second order by imposing additional regularity hypotheses on the time-step of the discretization.
For an in-depth discussion on the stroboscopy technique and its applications to partial differential equations, we remand to \cite{imme1,petite,sari}.

A delicate point in the time discretization of PDEs is that the discrete time step cannot be chosen arbitrarily. In fact, it is often the case that the time-step of the discretization must be chosen in accord to some bounds that depend upon the specific problem.
As an example, consider the nonstandard model for the heat equation discussed in \cite{watt}, where the time-step is dependent upon the diameter of the grid and upon the diffusion coefficients.
%In the paper, the coefficients are assumed to be constant; however, if these coefficients are not regular enough, then the finite difference in time does not generalize faithfully the partial difference in time.
In general, if the discrete timeline $\mathbb{T}$ is a deformation of the grid $\Lambda$, then the finite difference in time does not generalize faithfully the partial difference in time, and Theorem \ref{teorema equivalenza derivate} fails.
However, it is possible to determine sufficient conditions over $\mathbb{T}$ that imply the existence of $k \in \N$ such that Theorem \ref{teorema equivalenza derivate} holds for derivatives up to order $k$.
This study has been carried out in depth by van den Berg in \cite{imme2}.

Another possible approach to time-dependent PDEs is to follow the idea of Capi\'{n}sky and Cutland in \cite{capicutland1, capicutland statistic} and subsequent works: the authors did not discretize in time, but instead worked with functions defined on $\hR \times \Lambda^k$, where the first variable represents time, and the other $k$ variables represent space.
Following this idea, we formulate the problem
\begin{equation}\label{time dependent standard}
u_t - Fu = f
\end{equation}
with $u : \R \rightarrow V \subseteq L^2(\Omega)$, $F : V \rightarrow W \subseteq L^2(\Omega)$ by the grid function problem
\begin{equation}\label{time dependent grid}
u_t - \F u = f_\Lambda
\end{equation}
with $u : \hR \rightarrow \grid{\dom}$, with $[f_\Lambda] = f$, and where $\F$ is a suitable extension of $F$ in the sense of Theorems \ref{teorema lift lineare} and \ref{teorema lifting nonlineari}.
Notice that, by Theorem \ref{teorema equivalenza derivate2} and by Theorem \ref{teorema lift lineare}, the grid function formulation of a time dependent PDE is formally a hyperfinite system of ordinary differential equations, and it can be solved by exploiting the standard theory of dynamical systems.

Once we have a grid function formulation for a time dependent PDE, we would like to study the relation between its solutions and the solutions to the original problem.
%As expected, if for a suitable choices of $\F$ problem \ref{time dependent grid} has a solution $u$ and $u$ is regular enough, then $u$ induces a classical solution to problem \ref{time dependent standard}.
If $\norm{u(t)}_\infty$ is finite and uniformly bounded in $t$, by the same argument of Theorem \ref{young} $u$ corresponds to a Young measure $\nu^u : [0,T] \times \Omega \rightarrow \prob(\R)$.
If the composition $F(\nu^u)$ is defined in the sense of equation \ref{young equivalence equation}, then $\nu^u$ satisfies the equality
\begin{eqnarray*}
	\int_{[0,T]\times\Omega} \int_{\R} \tau d\nu^u(t,x) \varphi_t + \int_{\R}F(\tau)d\nu^u(t,x) \varphi d(t,x) + \\
	+ \int_{\Omega} \int_{\R} \tau d\nu^u(0,x)\varphi(0,x) dx 
	&=& \int_{[0,T]\times\Omega} [f_\Lambda]\varphi d(t,x)
\end{eqnarray*}
for all $\varphi \in C^1([0,T],\tests(\Omega))$ with $\varphi(T,x) = 0$.
If $u$ is more regular, the above equality can be exploited to prove that $[u]$ is a weak, strong or classical solution to the original problem, in the same spirit as in Theorem \ref{teorema dirichlet}.
See also the discussion in Remark \ref{remark estensioni}.

If $\norm{u(t)}_\infty$ is not finite, the sense in which $[u]$ is a solution to problem \ref{time dependent standard} has to be addressed on a case-by-case basis.
In \cite{illposed}, we will discuss an example where $\norm{u(t)}_1$ is finite and uniformly bounded in time, and $[u]$ can be interpreted as a Radon measure solution to problem \ref{time dependent standard}.

\section{Selected applications}\label{selected applications}

We believe that the grid functions are a very general theory that provide a unifying approach to a variety of problems from different areas of functional analysis.
To support our claim, we will use the theory of grid functions to study two classic problems from functional analysis that are usually studied with very different techniques: the first problem concerns the nonlinear theory of distributions, and the second is a minimization problem from the calculus of variations.
%These examples are meant to show how grid functions can be applied to a variety of problems while retaining coherence with the various standard approaches.
For an in-depth discussion of a grid function formulation of a class of ill-posed PDEs, we refer to \cite{illposed}.

\subsection{The product $HH'$}\label{section hh}

The following example is discussed in the setting of Colombeau algebras in \cite{colombeau advances}, and it can also be formalized in the framework of algebras of asymptotic functions \cite{oberbuggenberg}.

Let $H$ be the Heaviside function
$$
H(x) = \left\{
\begin{array}{ll}
0 & \text{if } x \leq 0\\
1 & \text{if } x > 0
\end{array}	
\right.
$$
and let $H'$ be the derivative of the Heaviside function in the sense of distributions, i.e.\ the Dirac distribution centered at $0$.
It is well-known that the product $HH'$ is not well-defined in the sense of distributions.
However, this product arises quite naturally in the description of some physical phenomena.
For instance, in the study of shock waves discussed in \cite{colombeau advances}, it is convenient to treat $H$ and $H'$ as a smooth functions and performing calculations such as
\begin{equation}\label{silly}
\int_{\R} (H^m-H^n)H'dx
=
\left[ \frac{H^{m+1}}{m+1} \right]^{+\infty}_{-\infty} - \left[ \frac{H^{n+1}}{n+1} \right]^{+\infty}_{-\infty}
=
\frac{1}{m+1}-\frac{1}{n+1}.
\end{equation}
This calculation is not justified in the theory of distributions: on the one hand, $H^m=H^n$ for all $m, n\in\N$, so that we intuitively expect that the integral should equal $0$; on the other hand, since the products $H^mH'$ and $H^nH'$ are not defined, the integrand is not well-defined.

We will now show how in the setting of grid functions one can rigorously formulate the integral \ref{silly} and compute the product $HH'$.
Let $M \in \ns{\N}\setminus\N$ satisfy $M\varepsilon \sim 0$, and consider the grid function $h\in\test'(\Lambda)$ defined by
$$
h(x) = \left\{
\begin{array}{ll}
0 & \text{if } x \leq 0\\
x/(M\varepsilon) & \text{if } 0 < x < M\varepsilon\\
1 & \text{if } x \geq M\varepsilon
\end{array}	
\right.
$$
The function $\D h$ is given by
$$
\D h(x) = \left\{
\begin{array}{ll}
0 & \text{if } x \leq 0 \text{ and } x \geq M\varepsilon\\
1/(M\varepsilon) & \text{if } 0 < x < M\varepsilon
\end{array}	
\right.
$$
In the next Lemma, we will prove that $h$ is a representative of the Heaviside function for which the calculation \ref{silly} makes sense.

\begin{lemma}\label{lemma prodotto}
	The function $h$ has the following properties:
	\begin{enumerate}
		\item $[h^m]=H$ and $[\D h^m]=\delta_0$ for all $m \in \ns{\N}$;
		\item $h^m \not = h^n$ whenever $m \not = n$;
		\item $\langle h^m-h^n, \D h \rangle \sim \frac{1}{m+1}-\frac{1}{n+1}$.
	\end{enumerate}
\end{lemma}
\begin{proof}
	(1).
	Let $\varphi \in \test(\Lambda)$ and, without loss of generality, suppose that $\varphi(x) \geq 0$ for all $x \in \Lambda$.
	Then for all $m \in \ns{\N}$ we have the inequalities
	$$
	\varepsilon \sum_{x \geq M\varepsilon} \varphi(x)  \leq \langle h^m, \varphi \rangle \leq \varepsilon \sum_{x \geq 0} \varphi(x),
	$$
	and, by taking the standard part of all the sides of the inequalities, we deduce
	$$
	\int_{0}^{+\infty} \sh{\varphi(x)} dx \leq \sh{\langle h^m, \varphi \rangle} \leq \int_{0}^{+\infty} \sh{\varphi(x)} dx.
	$$
	This is sufficient to conclude that $[h^m]=H$ for all $m \in \ns{\N}$.
	By Theorem \ref{teorema equivalenza derivate2}, $[\D h^m] = H' = \delta_0$.
	
	(2).
	Let $m \not = n$.
	Then,
	$$
	(h^m-h^n)(x) = \left\{
	\begin{array}{ll}
	0 & \text{if } x \leq 0 \text{ and } x \geq M\varepsilon\\
	(x/(M\varepsilon))^m-(x/(M\varepsilon))^n & \text{if } 0 < x < M\varepsilon.
	\end{array}	
	\right.
	$$
	In particular, $h^m - h^n \not = 0$, even if $[h^m] - [h^n] = 0$.
	
	(3).
	By the previous point,
	$$
	\langle h^m-h^n, \D h \rangle
	=
	\frac{1}{M} \sum_{j = 1}^{M} (j/M)^m-(j/M)^n.
	$$
	Since $M$ is infinite, 
	$$
	\frac{1}{M} \sum_{j = 1}^{M} (j/M)^m-(j/M)^n
	\sim
	\int_0^1 x^m-x^n dx
	=
	\frac{1}{m+1}-\frac{1}{n+1}.
	$$
\end{proof}

Thanks to the lemma above, we can compute the equivalence class in $\tests'(\R)$ of the product $hh'$.

\begin{corollary}\label{corollario prodotto}
	$[h \D h] = \frac{1}{2} H'$.
\end{corollary}
\begin{proof}
	For any $\varphi \in \test(\Lambda)$, we have
	$$\langle h \D h, \varphi \rangle
	=
	\frac{1}{M^2} \sum_{j = 1}^M j \psi(j\varepsilon).$$
	Let
	$
	\underline{m} = \min_{1 \leq j \leq M}\{\varphi(j\varepsilon)\}$ and $\overline{m} = \max_{1 \leq j \leq M}\{\varphi(j\varepsilon)\}$.
	We have the following inequalities:
	\begin{eqnarray*}
		\frac{\underline{m}}{M} \sum_{j = 1}^M j/M
		\leq
		\frac{1}{M^2} \sum_{j = 1}^M j \varphi(j\varepsilon)
		\leq
		\frac{\overline{m}}{M} \sum_{j = 1}^M j/M.
	\end{eqnarray*}
	Since $M$ is infinite,
	$$
	\frac{1}{M} \sum_{j = 1}^M j/M \sim \int_0^1 x dx=\frac{1}{2},
	$$
	so that
	$$
	\sh{\left(\frac{\underline{m}}{2}\right)} \leq \sh{\langle h \D h, \varphi \rangle} \leq \sh{\left(\frac{\overline{m}}{2}\right)}.
	$$
	By S-continuity of $\varphi$, $\underline{m} \sim \overline{m} \sim \varphi(0)$, so that $\sh{\langle h \D h, \varphi \rangle } = \frac{1}{2} \,\sh{\varphi(0)}$ for all $\varphi \in \test(\Lambda)$, which is equivalent to $[h \D h]=\frac{1}{2} H'$.
\end{proof}

Notice that $h$ is not the only function satisfying Lemma \ref{lemma prodotto} and Corollary \ref{corollario prodotto}.
In fact, we conjecture that Lemma \ref{lemma prodotto} and Corollary \ref{corollario prodotto} hold for a class of grid functions that satisfy some regularity conditions yet to be determined.

\subsection{A variational problem without a minimum}

We will now discuss a grid function formulation of a classic example of a variational problem without a minimum.
For an in-depth analysis of the Young measure solutions to this problem we refer to \cite{sychev}, and for a discussion of a similar problem in the setting of ultrafunctions, we refer to \cite{ultraapps}.
The grid function formulation consists in a hyperfinite discretization, as in Cutland \cite{cutland controls2}.

Consider the problem of minimizing the functional
\begin{equation}\label{variational}
J(u) = \int_0^1 \left( \int_0^x f(t)dt\right)^2  + (f(x)^2-1)^2 dx
\end{equation}
with $f \in L^2([0,1])$.
Intuitively, a minimizer for $J$ should have a small mean, but nevertheless it should assume values in the set $\{-1, +1\}$.
Let us make precise this idea: define
$$
f_0 = \chi_{[k,k+1/2)}-\chi_{[k+1/2,k+1)},\ k \in \Z
$$	
and let $f_n : [0,1] \rightarrow \R$ be defined by $f_n(x) = f_0(nx)$.
It can be verified that $\{f_n\}_{n \in \N}$ is a minimizing sequence for $J$, but $J$ has no minimum.
However, the sequence $\{f_n\}_{n \in \N}$ is uniformly bounded in $L^\infty([0,1])$, hence it admits a weak* limit in the sense of Young measures.
The limit is given by the constant Young measure %$\nu : [0,1] \rightarrow ???$ defined by
$$
\nu_x = \frac{1}{2} (\delta_1 + \delta_{-1}).
$$
%For all $f \in C^0(\R)$, the composition $f(\nu)$ is defined as follows:
%$$f(\nu)(x) = \int_\R f d\nu(x) = \frac{1}{2}(f(1)+f(-1)).$$
We can evaluate $J(\nu)$: for the first term of the integral \ref{variational}, we have
$$
\int_0^x \nu(t) dt = \int_0^x \left(\int_\R \tau d\nu_x\right) dt = 0,
$$
meaning that the barycentre of $\nu$ is $0$.
Since the support of $\nu$ is the set $\{-1, +1\}$, the second term of the integral becomes
$$
(\nu(x)^2-1)^2 = \int_\R (\tau^2-1)^2 d\nu_x = 0.
$$
As a consequence, $J(\nu) = 0$, and $\nu$ can be interpreted as a minimum of $J$ in the sense of Young measures.

In the setting of grid functions, the functional \ref{variational} can be represented by
$$
\J(u) = \varepsilon \sum_{n=0}^N \left[\left(\varepsilon \sum_{i = 0}^n f(i\varepsilon)\right)^2 + (f(n\varepsilon)^2-1)^2\right].
$$
Observe that this representation is coherent with the informal description of $J$, and that the only difference between $J$ and $\J$ is the replacement of the integrals with the hyperfinite sums.
Let us now minimize $\J$ in the sense of grid functions.
The minimizing sequence found in the classical case suggests us that a minimizer of $\J$ should assume values $\pm 1$, and that it should be piecewise constant in an interval of an infinitesimal length.
For $M \in \ns{\N}$, let $f_M = \ns{u_0(Mx)}$.
If $M < M' \leq N/2$, then
$$
\varepsilon \sum_{i = 0}^n f_M(i\varepsilon) > \varepsilon \sum_{i = 0}^n f_{M'}(i\varepsilon).
$$
We deduce that a minimizer for $\J$ is the grid function $f_{N/2}$, that is explicitly defined by $f_{N/2}(n\varepsilon) = (-1)^n$.

We will now show that this solution is coherent with the one obtained with the classic approach, i.e.\ that the Young measure associated to $f_{N/2}$ corresponds to $\frac{1}{2} (\delta_1 + \delta_{-1})$.
Since $\norm{f_{N/2}}_\infty = 1$, Theorem \ref{young} guarantees the existence of a Young measure $\nu$ that corresponds to $f_{N/2}$.
Moreover, by Proposition \ref{homogeneous}, $\nu$ is constant, and
$$
\int_{\R} \Psi d\nu_x = \frac{1}{2} \sum_{i = 0}^1 \Psi(f_{N/2}(i\varepsilon)) = \frac{1}{2} (\Psi(1)+\Psi(-1))
$$
for all $\Psi \in C^0_b(\R)$.
We deduce that the Young measure associated to $f_{N/2}$ is constant and equal to $\frac{1}{2} (\delta_1 + \delta_{-1})$, the minimizer of $J$ in the sense of Young measures.

\end{document}